\documentclass[12pt]{amsart}
\usepackage{amsmath,amssymb,a4wide,graphicx}
\usepackage{color}
\allowdisplaybreaks

\newtheorem{theorem}{Theorem}

\newtheorem{lemma}[theorem]{Lemma}
\newcommand{\R}{\mathbb R}
\newcommand{\N}{{\mathbb N}}
\newcommand{\T}{{\mathbb T}}
\let\pa\partial
\let\na\nabla
\let\eps\varepsilon
\newcommand{\diver}{\mbox{\rm div}}

\DeclareMathOperator*{\argmin}{argmin}

\begin{document}
\title[A higher-order gradient flow scheme]{A higher-order gradient flow scheme 
for a singular one-dimensional diffusion equation}

\author[B. D\"uring]{Bertram D\"uring}
\address{Department of Mathematics, University of Sussex, Pevensey II,
Brighton BN1 9QH, United Kingdom}
\email{b.during@sussex.ac.uk}
\author[P. Fuchs]{Philipp Fuchs}
\address{Institute for Analysis and Scientific Computing,
Vienna University of Technology, Wiedner Hauptstr. 8-10, 1040 Wien, Austria}
\email{philipp.fuchs@tuwien.ac.at}
\author[A. J\"ungel]{Ansgar J\"ungel}
\address{Institute for Analysis and Scientific Computing,
Vienna University of Technology, Wiedner Hauptstr. 8-10, 1040 Wien, Austria}
\email{juengel@tuwien.ac.at}

\thanks{The last two authors acknowledge partial support from   
the Austrian Science Fund (FWF), grants P24304, P27352, and W1245} 

\begin{abstract}
A nonlinear diffusion equation, interpreted as a Wasserstein gradient flow, 
is numerically solved in one space dimension
using a higher-order minimizing movement scheme based on the
BDF (backward differentiation formula) discretization. In each time step,
the approximation is obtained as the solution of a constrained quadratic
minimization problem on a finite-dimensional space consisting of
piecewise quadratic basis functions.
The numerical scheme conserves the mass and dissipates the
$G$-norm of the two-step BDF time approximation. Numerically, also the discrete 
entropy and variance are decaying. 
The decay turns out to be exponential in all cases. 
The corresponding decay rates are computed numerically for various grid numbers.
\end{abstract}


\maketitle

%%%%%%%%%%%%%%%%%%%%%%%%%%%%%%%%%%%%%%%%%%%%%%%%%%%%%%%%%%%%%%%%%%%%%%%%%%%%%%

\section{Introduction}

The aim of this paper is to propose and study a fully discrete higher-order
variant of the minimizing movement scheme in one space dimension 
for the nonlinear diffusion equation
\begin{equation}\label{1.eq}
  \pa_t u = \alpha^{-1}\Delta(u^\alpha) \quad\mbox{in }\T^d,\ t>0, \quad u(0)=u^0,
\end{equation}
with negative exponent $\alpha<0$, where $\T^d$ is the $d$-dimensional
torus. 
Equation~\eqref{1.eq} can be written as the gradient flow of the entropy 
$S[u]=(\alpha(\alpha-1))^{-1}\int_{\T^d}u^\alpha dx$ with respect to the
Wasserstein distance.
The gradient flow formulation gives rise to a natural semi-discretization in time
of the evolution by means of the minimizing movement scheme, leading to 
a minimization problem for the sum of kinetic and potential energies.

Minimizing movement schemes for evolution equations with an underlying
gradient flow structure were first suggested by De Giorgi \cite{DeG93}
in an abstract framework. 
Jordan, Kinderlehrer, and Otto \cite{JKO98} have shown that the solution
to the linear Fokker-Planck equation can be obtained by minimizing the 
logarithmic entropy
in the Wasserstein space. Since then, many nonlinear evolution equations have been
shown to constitute Wasserstein gradient flows, for instance the porous-medium
equation \cite{Ott01}, the Keller-Segel model \cite{BCC08}, 
equations for interacting gases \cite{CMV06}, and a nonlinear fourth-order equation
for quantum fluids \cite{GST09}.

The minimizing movement scheme can be interpreted as
an implicit Euler semi-discre\-ti\-za\-tion 
with respect to the Wasserstein gradient flow
structure. While for some time this has been mainly used as an analytical
tool, more recently numerical approximations of evolution equations
using this scheme have been proposed.
In one space dimension, the optimal transport metric becomes flat when
re-parametrized by means of inverse cumulative functions, which simplifies
the numerical solution; see e.g.\ \cite{AgBo13,BCC08,MaOs14,Osb15}. 
For multi-dimensional
situations, one approach is based on the Eulerian representation of the
discrete solution on a fixed grid. The resulting problem can be solved by using
interior point methods \cite{BFS12}, finite elements \cite{BCW10}, 
or finite volumes \cite{CCH15}. Another approach employs the Lagrangian 
representation, which is well adapted to optimal transport. 
Examples are explicit marching schemes \cite{GoTo06}, 
moving meshes \cite{BCW13}, linear finite elements for a fourth-order
equation \cite{DMM10}, reformulations in terms of evolving diffeomorphisms
\cite{CaMo09}, and entropic smoothing using the Kullback-Leibler
divergence \cite{Pey15}. 
The connection between Lagrangian schemes and the gradient flow structure
was investigated in \cite{KiWa99}.
In this paper, we will use the Lagrangian viewpoint.

The minimizing movement scheme of De Giorgi 
is of first order in time only since it is
based on the implicit Euler method. Concerning higher-order schemes, we are only
aware of the paper \cite{WeWi10}. There, second-order gradient flow schemes were
suggested for the Euler equations, 
with finite differences in space and the two-step BDF 
(Backward Differentiation Formula) method or diagonally implicit Runge-Kutta
(DIRK) schemes in time. 

In this paper, we propose a fully discrete second-order 
minimizing movement scheme using quadratic finite elements in space and the two-step
BDF method in time. We consider periodic point-symmetric solutions.
The finite-dimensional minimization problem, constrained
by the mass conservation, is solved by the method of Lagrange multipliers
which leads to a sequential quadratic programming problem.

By construction, our numerical scheme is of second order both in time and space,
it conserves the mass and dissipates the
$G$-norm of the approximation $\mathbf{g}^k$ of the quadratic ansatz functions
at time step $k$, 
\begin{equation}\label{1.Gnorm}
  \|(\mathbf{g}^{k+1},\mathbf{g}^k)\|_G^2 
	= \frac52(\mathbf{g}^{k+1})^\top M_w \mathbf{g}^{k+1}
	- 2(\mathbf{g}^{k+1})^\top M_w\mathbf{g}^{k} 
	+ \frac12(\mathbf{g}^{k})^\top M_w\mathbf{g}^{k},
\end{equation}
where the matrix $M_w$ is defined in the approximation of the Wasserstein metric 
on the space of the quadratic ansatz functions. We refer to Section~\ref{sec.disc}
for details. It turns out that numerically, the relative $G$-norm decays
exponentially fast to zero. Although we cannot expect for the multistep scheme
that the discrete entropy decays exponentially fast, this holds true for the
numerical experiments performed in this paper. Furthermore, also the discrete 
variance of the original variable $u$ and the Lagrangian variable decay 
exponentially fast. 

The numerical tests also indicate that the decay rate of the entropy is
increasing in the number of grid points, i.e., the discrete decay rates
are smaller than the (expected) value for the continuous equation.
This result is in accordance with the findings
of \cite{Mie13} for a finite-volume approximation of a one-dimensional 
linear Fokker-Planck equation.

Let us briefly review the literature for the
diffusion equation \eqref{1.eq} with $\alpha<0$. Equation~\eqref{1.eq} 
with $\alpha=-1$ appears in the modeling of heat conduction in solid Helium, 
where the solution $u$ corresponds to the inverse temperature \cite{Ros79}.
When this equation is considered in the whole space
or in a bounded domain with homogeneous Dirichlet boundary conditions, it is 
sometimes called the super-fast diffusion equation \cite[Chap.~9]{Vaz06}.
The critical exponent of this equation in one space dimension 
is $\alpha=-1$. For $\alpha>-1$ and $u^0\in L^1(\R)$, we have a smoothing property,
namely $u(t)\in L^\infty(\R)$ for any $t>0$ \cite[Section~9.1]{Vaz06}. 
For $\alpha\le -1$, no solutions exist with data in $L^1(\R)$. The non-existence
range in dimensions $d\ge 2$ contains even all negative exponents, 
$\alpha<0$ \cite{Vaz92}. However, if
$d=1$ and $\alpha\le -1$, there is a weak smoothing effect. Indeed, given
$u^0\in L_{\rm loc}^p(\R)$ for $p=(1-\alpha)/2$, the solution $u$ exists 
and is locally bounded in $\R\times(0,\infty)$. Furthermore, if $u^0\in L^p(\R)$ 
then there is instantaneous extinction, i.e.\ $u(t)=0$ in $\R$ for all $t>0$
\cite[Theorem~9.3]{Vaz06}.

Clearly, such results cannot be expected when the super-fast diffusion
equation is considered on the torus. 
We expect that global-in-time weak solutions exist, 
which converge to the constant steady state as $t\to\infty$.
Since we could not find any results on the existence and large-time asymptotics 
in the literature in that situation and since the use of negative exponents
is less standard, we provide a (short) proof for completeness
in the appendix; see Section~\ref{sec.ex} for details.

The paper is organized as follows. The global existence result and the
exponential decay of the solutions to the constant steady state as well as
some basic properties of the Wasserstein distance for periodic functions are stated
in Section~\ref{sec.pre}. Section~\ref{sec.disc} is devoted to the description
of the numerical scheme, and some numerical experiments are presented in
Section~\ref{sec.num}. We conclude in Section \ref{sec.conc}.
The appendix contains the proofs of the existence and
large-time asymptotics theorems and the calculations of the coefficients of
the matrix $M_w$ and the Hessian of the discrete entropy.

%%%%%%%%%%%%%%%%%%%%%%%%%%%%%%%%%%%%%%%%%%%%%%%%%%%%%%%%%%%%%%%%%%%%%%%%%%

\section{Prerequisites}\label{sec.pre}

\subsection{Existence of solutions and large-time asymptotics}\label{sec.ex}

Equation \eqref{1.eq} on the torus $\T^d$ does not possess the non-existence 
or instantaneous extinction properties of the super-fast diffusion equation in
the whole space
since mass cannot get lost. In fact, we expect that for any $\alpha<0$, there
exists a global weak solution. If the initial datum $u^0$ is
nonnegative only, \eqref{1.eq} is still a singular diffusion equation.
However, because of the fast diffusion, the solution becomes positive
for all positive times and, by parabolic regularity theory, also smooth.

\begin{theorem}[Existence of weak solutions]\label{thm.ex}
Let $\alpha<0$ and let $u^0\in L^\infty(\T^d)$ satisfy $u^0\ge 0$. 
If $\alpha=-1$, we assume additionally that $\int_{\T^d}\log u^0 dx>-\infty$. 
Then there exists a unique weak
solution to \eqref{1.eq} satisfying $u^\alpha\in L^2(0,T;H^1(\T^d))$,
$\pa_t u \in L^2(0,T;H^1(\T^d)')$ for all $T>0$, and $0\le u(x,t)\le
\sup_{\T^d} u^0$ for $x\in\T^d$, $t\ge 0$.
\end{theorem}

The proof of this theorem is based on a standard regularization procedure
but we need to distinguish carefully the cases $-1<\alpha<0$, $\alpha=-1$,
and $\alpha<-1$.
We present the (short) proof for completeness in Appendix \ref{app.ex}.
For $t\to\infty$, the (smooth) solution converges to the constant steady state.
Since this constant is positive, diffusion slows down when time increases.
Therefore, we cannot expect instantaneous extinction phenomena. Still,
we are able to prove that the convergence is exponentially fast with respect to the
$L^1$-norm. We introduce the entropy
$$
  H_\beta[u] = \int_{\T^d}u^\beta \,dx - \bigg(\int_{\T^d}u\,dx\bigg)^\beta, 
	\quad \beta>1.
$$
The steady state of \eqref{1.eq} is given by $u_\infty=\int_{\T^d} u^0\, dx$
if $\mbox{vol}(\T^d)=1$.

\begin{theorem}[Exponential decay]\label{thm.time}
Let $u$ be a smooth positive solution to \eqref{1.eq} and let
$\text{\rm vol}(\T^d)=1$. Then
$$
  \|u(t)-u_\infty\|_{L^1(\T^d)} \le C_\beta H_\beta[u^0]^{1/\beta}
	\|u^0\|_{L^1(\T^d)}^{1/2}e^{-\lambda t},
$$
%$u(t)$ converges to $u_\infty$ exponentially fast in
%the following sense:
%$$
%  H_\beta[u(t)] \le H_\beta[u^0]e^{-\lambda t}, \quad t\ge 0,
%$$
where $C_\beta>0$ and for $\alpha<0$, $1<\beta\le 2$,
$$
  \lambda = \frac{2(\beta-1)}{\beta C_B}(\sup_{\T^d} u^0)^{\alpha-1}, 
$$
and $C_B>0$ is the constant in the Beckner inequality \eqref{app.bec};
for $-1\le\alpha<0$, $\beta=2(1-\alpha)$,
$$
	\lambda = \frac{2(1-2\alpha)}{(1-\alpha)\|u^0\|_{L^1(\T^d)}^{1-\alpha}}.
$$
\end{theorem}

In the first result, the decay rate $\lambda$ depends on $\sup_{\T^d}u^0$, which
seems to be not optimal. The decay rate in the second result depends on
the $L^1$-norm of $u^0$ only but we need a particular value of $\beta$.
The proof is based on the entropy method; see Appendix~\ref{app.ex}.
Stronger decay results have been derived for the fast-diffusion
equation in the whole space or in bounded domains; see, e.g.,
\cite{BDGV10,Vaz06}. However, our proof is very elementary and just an illustration
for the qualitative behavior of the solutions to \eqref{1.eq}.

%%%%%%%%%%%%%%%%%%%

\subsection{The Wasserstein distance for periodic functions}\label{sec.wass}

We recall some basic facts about mass transportation and Wasserstein distances
following \cite{DMM10}.
For more details, we refer to \cite{AGS05,Vil09}. Let $X$ be a Riemannian
manifold with distance $d:X\times X\to\R_+:=[0,\infty)$ and let $P_M(X)$ be the
convex set of measures with fixed mass $M>0$ on $X$.
The $L^2$-Wasserstein distance of two measures $\mu_1$, $\mu_2\in P_M(X)$ is
defined by
\begin{equation}\label{pre.W}
  W[\mu_1,\mu_2]^2 = \inf_{\pi\in\Pi(\mu_1,\mu_2)}\int_{X\times X}d(x,y)^2
	\,d\pi(x,y),
\end{equation}
where $\Pi(\mu_1,\mu_2)$ is the set of transport plans connecting $\mu_1$ 
with $\mu_2$, i.e.\ the set of all measures on $X\times X$ with respective marginals
$\mu_1$ and $\mu_2$, 
$$
  \pi(A\times X) = \mu_1(A), \quad \pi(X\times B) = \mu_2(B)
$$
for all measurable sets $A$, $B\subset X$. If the measures $\mu_1$ and $\mu_2$
possess densities $u_1$ and $u_2$, respectively, with respect to a fixed 
measure on $X$, then we write, slightly abusing the notation, $W[u_1,u_2]$
instead of $W[\mu_1,\mu_2]$. 
%For later reference, we remark that the infimum in
%\eqref{pre.W} is indeed a minimum, realized by some optimal plan
%$\pi_{\rm opt}\in \Pi(\mu_1,\mu_2)$; see \cite[Theorem~6.2.4]{AGS05}.

In one space dimension, there exists an explicit formula to compute $W$.
Let $X=(a,b)\subset\R$ be a (possibly infinite) interval, and let $\mu_1$,
$\mu_2\in P_M(X)$ be two measures. We define their distribution functions 
$$
  U_i:(a,b)\to[0,M],\quad U_i(x)=\mu_i((a,x]),\quad i=1,2.
$$
As these functions are right-continuous and monotonically increasing, they possess
right-continuous increasing pseudo-inverse functions $G_i:[0,M]\to[a,b]$,
given by
$$
  G_i(\omega) = \inf\{x\in(a,b):U_i(x)>\omega\}, \quad i=1,2.
$$
Then \cite{Vil03},
\begin{equation}\label{pre.W2}
  W[u_1,u_2]^2 = \int_0^M(G_1(\omega)-G_2(\omega))^2 d\omega.
\end{equation}

This formula does not extend to $X=\T\simeq(0,1)$ because of
the topology induced by the periodic boundary conditions,
$d(x,y)=\min\{|x-y|,1-|x-y|\}$. The reason is that mass can be transported
either clock- or counter-clockwise (see \cite{DMM10} for details). However,
if the densities $u_1$ and $u_2$ are point-symmetric, \eqref{pre.W2} still
holds. More precisely, let $u_i(x)=u_i(1-x)$ for $x\in(0,1)$ and $i=1,2$.
Then \eqref{pre.W2} holds, where $G_i:[0,M]\to[0,1]$ is the
inverse function of $U_i(x) = \int_0^x u_i(y)dy$ \cite[Lemma 2.2]{DMM10}.

%%%%%%%%%%%%%%%%%%%%%%%%%%%%%%%%%%%%%%%%%%%%%%%%%%%%%%%%%%%%%%%%%%%%%%%%%%

\section{Time discretization and Lagrangian coordinates}\label{sec.disc}

\subsection{The semi-discrete BDF scheme}\label{sec.BDF}

We introduce the second-order minimizing movement scheme. 
First, we explain the underlying idea for the finite-dimensional gradient flow
\begin{equation}\label{disc.x}
  \dot x = -\na\phi(x), \quad t>0, \quad x(0)=x_0,
\end{equation}
where $\phi:\R^d\to\R$ is a smooth potential. This equation can be approximated
by the following minimization problem:
$$
  x^{n+1} = \argmin_{x\in\R^d}\Phi(x), \quad
	\Phi(x) = \frac{1}{2\tau}\|x-x^n\|^2 + \phi(x),
$$
where $\tau>0$ is the time step size and $x^n$ is an approximation of
$x(n\tau)$. The minimizer $x^{n+1}$ is a critical point and thus,
$$
  0 = \na\Phi(x^{n+1}) = \frac{1}{\tau}(x^{n+1}-x^n) + \na\phi(x^{n+1}),
$$
which corresponds to the implicit Euler scheme. 

Instead of the Euler scheme, we wish to discretize \eqref{1.eq} by a multistep
method. As an example, consider the two-step BDF (or BDF-2) method,
$$
  \frac{1}{\tau}\left(\frac32x^{n+2} - 2x^{n+1} + \frac12 x^{n}\right)
	= -\na\phi(x^{n+2}),
$$
where $x^n$ and $x^{n+1}$ are given. Writing this scheme as
$$
  \frac12(x^{n}-x^{n+2}) - 2(x^{n+1}-x^{n+2}) = -\tau\na\phi(x^{n+1}),
$$
we see that $x^{n+2}$ is a critical point of the functional
$$
  \Phi(x) = -\frac{1}{4\tau}\|x^n-x\|^2 + \frac{1}{\tau}\|x^{n+1}-x\|^2 + \phi(x).
$$

More generally, the BDF-$k$ approximation of \eqref{disc.x},
$$
  \sum_{i=0}^k a_ix^{n+i} = -\tau\na\phi(x^{n+k}),
$$
for given $x^n,\ldots,x^{n+k-1}$ can be formulated as 
$$
  -\tau\na\phi(x^{n+k}) = \sum_{i=0}^{k-1}a_ix^{n+i} + a_kx^{n+k}
	= \sum_{i=0}^{k-1}a_i(x^{n+i}-x^{n+k}),
$$
since $\sum_{i=0}^k a_i=0$, or as the minimization problem
\begin{equation}\label{disc.min}
  x^{n+k} = \argmin_{x\in\R^d}\Phi(x), \quad
	\Phi(x) = -\frac{1}{2\tau}\sum_{i=0}^{k-1}a_i\|x^{n+i}-x\|^2 + \phi(x),
\end{equation}
In a similar way, we may formulate general multistep methods. 
We recall that BDF-$k$ schemes are
consistent if $\sum_{i=0}^ka_i=0$, $\sum_{i=1}^k ia_i=1$ and 
zero-stable if and only if $k\le 5$ \cite[Section~11.5]{QSS07}. 

The same idea as above is applicable to gradient flows in the $L^2$-Wasserstein
distance. For this, we replace the $L^2$ norm by the Wasserstein distance.
For equation \eqref{1.eq}, scheme \eqref{disc.min} turns into
\begin{equation}\label{disc.bdf}
  u^{n+k} = \argmin_{u\in P_M(\T)}\Phi(u), \quad
	\Phi(u) = -\frac{1}{2\tau}\sum_{i=0}^{k-1}a_iW[u^{n+i},u]^2 + S[u],
\end{equation}
where $S[u]=(\alpha(\alpha-1))^{-1}\int_\T u^\alpha \, dx$. 
Scheme \eqref{disc.bdf} can be interpreted as a BDF-$k$ minimizing movement scheme.
%The connection between \eqref{disc.bdf} and the semi-discretization of 
%\eqref{1.eq} in the framework of the implicit Euler method was established
%formally in \cite{DMM10}.

%%%%%%%%%%%%%%%%%%%%

\subsection{Lagrangian coordinates}\label{sec.lagr}

Before introducing the spatial discretization, we rewrite the scheme 
\eqref{disc.bdf} in a more explicit manner, using the inverse distribution
functions $G$ and $G^*$ of $u$ and $u^*$, respectively, which were introduced
in Section \ref{sec.wass}. The numerical procedure is similar to that in
\cite{DMM10} but our higher-order scheme introduces some changes.
We call $\omega=U(x)\in[0,M]$ the Lagrangian coordinate, which is conjugate 
to the Eulerian coordinate $x\in\T$, and we refer to the inverse
distribution function $G$ as the associated Lagrangian map.
For a consistent change of variables, we need to express the entropy $S[u]$
in terms of the Lagrangian coordinates. With the formula for the
differential of an inverse function,
$$
  u(x) = \pa_x U(x) = \frac{1}{\pa_\omega G(\omega)},
$$
and the change of unknowns $x=G(\omega)$ under the integral in $S[u]$, we obtain
$$
  S[u] = \frac{1}{\alpha(\alpha-1)}\int_\T\frac{u(x)\, dx}{u(x)^{1-\alpha}}
	= \frac{1}{\alpha(\alpha-1)}\int_\T g(\omega)^{1-\alpha}\, d\omega,
$$
where $g(\omega)=\pa_\omega G(\omega)$. Note that the exponent in the
integrand is positive since $\alpha<0$. In terms of $g$, the expression for the 
Wasserstein distance in \eqref{pre.W2} becomes (see \cite[Section~2.3]{DMM10})
\begin{align*}
  W[u,u^*]^2 &= \int_0^M(G(\omega)-G^*(\omega))^2\, d\omega
	= \int_0^M\bigg(\int_0^\omega(g(\eta)-g^*(\eta))d\eta\bigg)^2
               \, d\omega \\
	&= \int_0^M\int_0^M(M-\max\{\eta,\eta'\})(g(\eta)-g^*(\eta))(g(\eta')-g^*(\eta'))
	\, d\eta\,d\eta'.
\end{align*}
This expression is simply a quadratic form in $g-g^*$ and thus easy to implement in 
the numerical scheme.
We summarize our results which slightly generalize Lemma 2.3 in \cite{DMM10}.

\begin{lemma}\label{lem.gu}
Let the initial datum $u^0:\T\to\R$ be point-symmetric. Then the solution
$u^{n+k}$ to the BDF-$k$ scheme \eqref{disc.bdf} is in one-to-one correspondence
to the solution $g^{n+k}$ obtained from the inductive scheme
\begin{equation}\label{disc.g}
  g^{n+k} = \argmin_g\Psi(g),
\end{equation}
with initial condition $g^0=1/u\circ G$ and given $g^1,\ldots,g^{n+k-1}$
obtained from a lower-order scheme. The argmin has to be taken over all
measurable functions $g:[0,M]\to(0,\infty)$ satisfying the mass constraint
$\int_0^M g(\omega)\,d\omega=1$, and the function $\Psi$ is given by
\begin{align}\label{disc.Psi}
  \Psi(g) &= -\frac{1}{2\tau}\int_0^M\int_0^M(M-\max\{\eta,\eta'\})\sum_{i=0}^{k-1}
	a_i(g(\eta)-g^{n+i}(\eta))(g(\eta')-g^{n+i}(\eta')) \, d\eta\,d\eta' \\
	&\phantom{xx}{}+ \frac{1}{\alpha(1-\alpha)}\int_0^M 
	g(\omega)^{1-\alpha}\, d\omega. \nonumber
\end{align}
Moreover, $\Phi(u)=\Psi(g)$ and the functions $u^n$ and $g^n$ are related by
\begin{equation}\label{disc.xom}
  u^n(x_\omega) = \frac{1}{g^n(\omega)}, \quad x_\omega =
  \int_0^\omega g^n(\eta)\, d\eta.
\end{equation}
\end{lemma}

In Lagrangian coordinates, the problem has become a minimization problem
in $\Psi(g)$ which is the sum of a quadratic form and a convex functional,
hence it is convex. In the special case $\alpha=-1$, the second integral
in \eqref{disc.Psi} is quadratic too which simplifies the numerical discretization.
Therefore, we will consider mainly numerical examples with $\alpha=-1$
in Section \ref{sec.num}.

%%%%%%%%%%%%%%%%%%%%%%%

\subsection{Spatial discretization}\label{disc.space}

We approximate the infinite-dimensional variational problem \eqref{disc.g} by
a finite-dimensional one. Minimization in \eqref{disc.g} is performed over
the finite-dimensional space of quadratic ansatz functions. 
This generalizes the approach in \cite{DMM10}, where only linear ansatz functions
were used. We define the ansatz space as follows.

Let $N\in\N$ and a mesh $\{x_0,\ldots,x_{N}\}$ on $[0,1]$ be given 
with $x_0=0$ and $x_{N}=1$. Using \eqref{disc.xom}, 
we construct the mesh $\Omega_N=\{\omega_0,\omega_1\ldots,\omega_{N}\}$ of
$[0,M]$. Then $\omega_0=0$, $\omega_{N}=M$,
$\omega_{i}<\omega_{i+1}$, and $\omega_{N-i}=M-\omega_{i}$ (point-symmmetry), 
where $i=1,\ldots,N-1$. 
Since we wish to introduce quadratic ansatz functions, we add the grid points 
$\omega_{j+1/2}=(\omega_{j+1}+\omega_{j})/2$ for $j=0,\ldots,N-1$.

The basis functions $\phi_j:\T\to\R$ are defined by
\begin{align*}
  \phi_j(\omega) &= \left\{\begin{array}{ll}
	\frac{\omega-\omega_{j-1}}{\omega_j-\omega_{j-1}}
	&\quad\mbox{for }\omega\in[\omega_{j-1},\omega_j], \\
	\frac{\omega_{j+1}-\omega}{\omega_{j+1}-\omega_{j}}
	&\quad\mbox{for }\omega\in[\omega_{j},\omega_{j+1}], \\
  0 &\quad\mbox{otherwise},
	\end{array}\right. \qquad j=1,\ldots,N-1, \\
	\phi_N(\omega) &= \left\{\begin{array}{ll}
	\frac{\omega_1-\omega}{\omega_1} &\quad\mbox{for }\omega\in[0,\omega_1], \\
	\frac{\omega-\omega_{N-1}}{M-\omega_{N-1}} 
	&\quad\mbox{for }\omega\in[\omega_{N-1},M], \\
	0 &\quad\mbox{otherwise}.
	\end{array}\right.
\end{align*}
This set of piecewise linear functions is supplemented by the following 
piecewise quadratic basis functions:
$$
  \phi_{N+j}(\omega) = \left\{\begin{array}{ll}
	1-\big(\frac{2\omega-(\omega_{j-1}+\omega_j)}{\omega_j-\omega_{j-1}}\big)^2
	&\quad\mbox{for }\omega\in[\omega_{j-1},\omega_j], \\
	0 &\quad\mbox{otherwise},
	\end{array}\right. \qquad j=1,\ldots,N.
$$
The ansatz space is the set of all positive, piecewise quadratic functions
$g:[0,M]\to\R_+$ of the form 
\begin{equation}\label{disc.gg}
  g(\omega)=\sum_{j=1}^{2N}g_j \phi_j(\omega).
\end{equation}
We call $\mathbf{g}:=(g_1,\ldots,g_{2N})\in[0,\infty)^{2N}$ 
the associated weight vector. 
By definition of $\phi_j$, we have $g(\omega_j)=g_j$ for $j=1,\ldots,N$.
Moreover, as $g$ is point-symmetric, $g_0=g_N$. 

Now, for given mass $M>0$ and grid $\Omega_N\subset[0,M]$, we define the
set ${\mathbb G}^N_M\subset\R_+^{2N}$ as the set of weight vectors $\mathbf{g}$
for which the associated interpolation $g$ from \eqref{disc.gg}
satisfies the mass constraint,
\begin{equation}\label{disc.mass}
  1 = \int_0^M g(\omega)\, d\omega = \sum_{j=1}^N \left(\frac{g_{j-1}+g_j}{2}
	+ \frac23 g_{N+j}\right)(\omega_j-\omega_{j-1}).
\end{equation}

The Wasserstein metric for functions approximated in this way becomes
\begin{align*}
  W[u,u^*]^2 &= \int_0^M\int_0^M(M-\max\{\eta,\eta'\})
	\sum_{i=1}^{2N}(g_i-g_i^*)\phi_i(\eta)
	\sum_{j=1}^{2N}(g_j-g_j^*)\phi_j(\eta')\, d\eta\,d\eta' \\
	&= \sum_{i,j=1}^N (g_i-g_i^*)(g_j-g_j^*)a_{ij}
	+ \sum_{i,j=1}^N (g_{N+i}-g_{N+i}^*)(g_j-g_j^*)b_{ij} \\
	&\phantom{xx}{}+ \sum_{i,j=1}^N (g_i-g_i^*)(g_{N+j}-g_{N+j}^*)b_{ji}
	+ \sum_{i,j=1}^N (g_{N+i}-g_{N+i}^*)(g_{N+j}-g_{N+j}^*)c_{ij}, 
\end{align*} 
where 
\begin{align}
  a_{ij} &= \int_0^M\int_0^M(M-\max\{\eta,\eta'\})\phi_i(\eta)\phi_j(\eta')
	\,d\eta\,d\eta', \nonumber \\
	b_{ij} &= \int_0^M\int_0^M(M-\max\{\eta,\eta'\})
	\phi_{N+i}(\eta)\phi_j(\eta')\,d\eta\,d\eta', \label{disc.abc} \\
	c_{ij} &= \int_0^M\int_0^M(M-\max\{\eta,\eta'\})
	\phi_{N+i}(\eta)\phi_{N+j}(\eta')\,d\eta\,d\eta'. \nonumber
\end{align}
The coefficients $a_{ij}$, $b_{ij}$, and $c_{ij}$ can be computed
explicitly. The explicit expressions are given in Appendix \ref{sec.coeff}.
Setting $A=(a_{ij})$, $B=(b_{ij})$, $C=(c_{ij})$, and defining the matrix
\begin{equation}\label{disc.M}
  M_w = (M_{ij}) = \begin{pmatrix} A & B^\top \\ B & C \end{pmatrix},
\end{equation}
we can formulate the above sum as
$$
  W[u,u^*]^2 = \sum_{i,j=1}^{2N} M_{ij}(g_i-g^*_i)(g_j-g^*_j).
$$
As the matrices $A$ and $C$ are symmetric, $M_w$ is symmetric, too. The matrix $A$
corresponds to the linear approximation considered in \cite{DMM10}.

%%%%%%%%%%%%%%%%%%%%%%%%%%%

\subsection{Minimization}\label{sec.min}

The numerical scheme consists of the following finite-dimensional variational
problem:
\begin{align}
  \mathbf{g}^{n+k} &= \argmin_{\mathbf{g}\in{\mathbb G}^N_M}
	\Psi_N(\mathbf{g}), \quad \label{disc.gfin} \\
\mbox{where}\quad	\Psi_N(\mathbf{g}) 
	&= -\frac{1}{2\tau}\sum_{\ell=0}^{k-1}a_\ell\sum_{i,j=1}^{2N} M_{ij}
	(g_i-g^{n+\ell}_i)(g_j-g^{n+\ell}_j) + S_N[\mathbf{g}], \nonumber
\end{align}
and where $\mathbf{g}=(g_1,\ldots,g_{2N})$,
\begin{equation}\label{disc.SN}
  S_N[\mathbf{g}] = \frac{1}{\alpha(\alpha-1)}
	\sum_{i=0}^{N-1}\int_{\omega_{i}}^{\omega_{i+1}}
	\big(g_i\phi_i + g_{i+1}\phi_{i+1} +
        g_{N+i}\phi_{N+i}\big)^{1-\alpha}\,  d\omega.
\end{equation}
The functions $\Psi_N(\mathbf{g})$ and $\Psi(g)$ from 
\eqref{disc.Psi} are related by $\Psi_N(\mathbf{g})=\Psi(g)$
with a piecewise quadratic function $g$ defined
from $\mathbf{g}$ by \eqref{disc.gg}.
Since $\Psi_N$ is convex for $\alpha<0$ and the set ${\mathbb G}^N_M$ is convex,
there exists a unique minimizer of \eqref{disc.gfin}.

%%%%%%%%%%%%%%%%%%%%%%

\subsection{Fully discrete Euler-Lagrange equations}\label{disc.EL}

The minimizer $\mathbf{g}_{n+k}$ in \eqref{disc.gfin} is subject to the
mass constraint \eqref{disc.mass}, by definition of the set $\mathbb{G}_M^N$.
Therefore, instead of working on the set $\mathbb{G}^N_M$, it is more convenient
to consider \eqref{disc.gfin} as a constrained minimization problem for
$\mathbf{g}$ on the larger set $\R^{2N}$, which is solved by the method
of Lagrange multipliers $\lambda$ using the Lagrange functional
$$
  L(\mathbf{g},\lambda) = \Psi_N(\mathbf{g})
	- \lambda\bigg(1-\sum_{j=1}^{N}\bigg(\frac{g_{j-1}+g_j}{2} 
	+ \frac23 g_{N+j}\bigg)(\omega_{j}-\omega_{j-1})\bigg).
$$
A critical point of $L$ satisfies the $2N$ conditions
$$
  0 = \mathbf{G}_j := \frac{\pa L}{\pa g_j}
	= -\frac{1}{\tau}\sum_{\ell=0}^{k-1}a_\ell\sum_{i=1}^{2N} M_{ij}(g_i-g_i^{n+\ell}) 
	+ \frac{\pa S_N}{\pa g_j}, \quad j=1,\ldots,2N.
$$
The precise values for $\pa S_N/\pa g_j$ are given in Appendix~\ref{sec.hess}
for $\alpha=-1$. The condition for the constraint is recovered from
$$
  0 = \mathbf{G}_{2N+1} := \frac{\pa L}{\pa\lambda}
	= 1-\sum_{j=1}^{N}\bigg(\frac{g_{j-1}+g_j}{2} 
	+ \frac23 g_{N+j}\bigg)(\omega_{j}-\omega_{j-1}).
$$
The vector $\mathbf{G}[\mathbf{g},\lambda]
=(\mathbf{G}_1,\ldots,\mathbf{G}_{2N+1})^\top\in\R^{2N+1}$
is the gradient of $L(\mathbf{g},\lambda)$ with respect to
$(\mathbf{g},\lambda)$. We approximate a critical point numerically by
applying the Newton method to the first-order optimality condition
$\mathbf{G}[\mathbf{g},\lambda]=0$. 
This leads to a sequential quadratic programming method, since at every 
Newton iteration step a quadratic subproblem has to be solved.

%%%%%%%%%%%%%%%%%%%%

\subsection{Implementation}\label{sec.impl}

Let the solution $\mathbf{g}$ at the $n$th time step be given and let
$\mathbf{g}^{(0)}:=\mathbf{g}$, $\lambda^{(0)}:=0$. The iteration
is as follows:
$$
  \mathbf{g}^{(s+1)} := \mathbf{g}^{(s)} + (\delta\mathbf{g})^{(s+1)}, \quad
  \lambda^{(s+1)} := \lambda^{(s)} + (\delta\lambda)^{(s+1)},
$$
where $((\delta\mathbf{g})^{(s+1)},(\delta\lambda)^{(s+1)})$ is the solution
to the linear system
$$
  H[\mathbf{g}^{(s)},\lambda^{(s)}]
	((\delta\mathbf{g})^{(s+1)},(\delta\lambda)^{(s+1)})^\top
	= -\mathbf{G}[\mathbf{g}^{(s)},\lambda^{(s)}],
$$
where $H[\mathbf{g}^{(s)},\lambda^{(s)}]$ denotes the Hessian of
$\Psi_N$, whose entries are given in Appendix \ref{sec.hess} for $\alpha=-1$. 
The iteration is stopped if the norm of 
$((\delta\mathbf{g})^{(s+1)},(\delta\lambda)^{(s+1)})$ is smaller
than a certain threshold (see Section \ref{sec.num} for details). 
In this case, we define
$\mathbf{g}^{n+1}:=\mathbf{g}^{(s+1)}$ and $\lambda^{n+1}:=\lambda^{(s+1)}$
at the $(n+1)$th time step.
For the BDF-$k$ scheme, the values $\mathbf{g}^1,\ldots,\mathbf{g}^{k-1}$ are 
computed from a lower-order scheme. In the numerical section below, we employ the
BDF-2 scheme only such that $\mathbf{g}^1$ is calculated by the implicit
Euler method.

Note that the constrained minimization problem is exactly mass conserving 
by construction, but the Newton iteration introduces a small error which depends
on the tolerance imposed in the Newton method.

In each iteration step, we need to invert the dense matrix $H$ which is the sum
of $M_w$ and the Hessian of $S_N$. This is not a numerical challenge in the
one-dimensional case we consider but it may become critical in multi-dimensional
discretizations on fine grids. For $\alpha=-1$, however, $M_w$ and the Hessian
of $S_N$ are constant matrices which significantly simplifies the Newton scheme.

%%%%%%%%%%%%%%%%%%%%%%%%%%

\subsection{Choice of the initial condition}\label{sec.init}

In order to compute the initial condition in Lagrangian coordinates, we need 
to make precise the values $g_i^0$ of the vector $\mathbf{g}\in\mathbb{G}_M^N$
and the points $x^0_j$ of the spatial lattice which is moving as the
solution evolves. Let the mesh $\{x_0^0,\ldots,x_N^0\}$ be given and set
$g^0(\omega_j)=1/u^0(x_j)$. Approximating the initial function by a linear ansatz
function, we obtain
\begin{equation}\label{disc.x0}
  x_j^0 = G^0(\omega_j) 
	= \frac12\sum_{i=1}^j (\omega_j-\omega_{j-1})(g^0_{i-1}+g^0_{i}), \quad
	j=1,\ldots,N.
\end{equation}
This is a system of linear equations in $\omega_1,\ldots,\omega_N$. Choosing
the uniform grid $x_j^0=j/N$, we can solve this system explicitly. Indeed,
since
$$
  \frac{1}{N} = x_{j+1}^0-x_j^0 = \frac12(\omega_{j+1}-\omega_j)(g^0_{j+1}+g^0_{j}),
$$
which can be solved for $\omega_{j+1}$:
$$
  \omega_{j+1} = \omega_j + \frac{2}{N}(g^0_{j+1}+g^0_{j})^{-1}, \quad
	j=0,\ldots,N-1.
$$
As $u^0$ is assumed to be point-symmetric, this is true for $(\omega_i)$ too.
Finally, we approximate $g^0(\omega_{j+1/2})$ by the arithmetic mean
$\frac12(g^0_{j-1}+g^0_j)$, $j=1,\ldots,N$. Consequently, the weights $g_{N+i}^0$
vanish for $i=1,\ldots,N$ in the expansion $g^0(\omega)=\sum_{j=0}^{2N}g_j^0\phi_j$,
which is consistent with our approximation \eqref{disc.x0}.
We also refer to the discussion in \cite[Section~2.8]{DMM10}.

%%%%%%%%%%%%%%%%%%%%%%%%%%%%%%%%%%%%%%%%%%%%%%%%%%%%%%%%%%%%%%%%%%%%%%%%%%

\section{Numerical experiments}\label{sec.num}

In this section, we present some numerical results for \eqref{1.eq} with $\alpha=-1$,
by employing the BDF-2 method with quadratic ansatz functions.
We choose a uniform grid for $x\in[0,1]$ with $N=100$ grid points, and
the time step size $\tau=10^{-5}$. The Newton iterations are stopped if both 
the relative $\ell^\infty$ error in the $g$-variables
and the $\ell^2$ norm of $\mathbf{G}[\mathbf{g}^{(s)},\lambda^{(s)}]$ 
are smaller than $10^{-8}$.

Figure \ref{fig.evol} illustrates the temporal evolution of the solution $u(x,t)$
with the initial conditions $u^0(x) = \cos(2\pi x)^2 + 0.01$ (left figure)
and $u^0(x) = \sqrt[5]{|x-0.5|+0.0001}-0.1$ (right figure).
We observe that, as expected, the solutions converge to the constant steady state.
Because of the negative exponent $\alpha$, very small initial values increase
quickly in time.

\begin{figure}[ht]
\includegraphics[width=75mm]{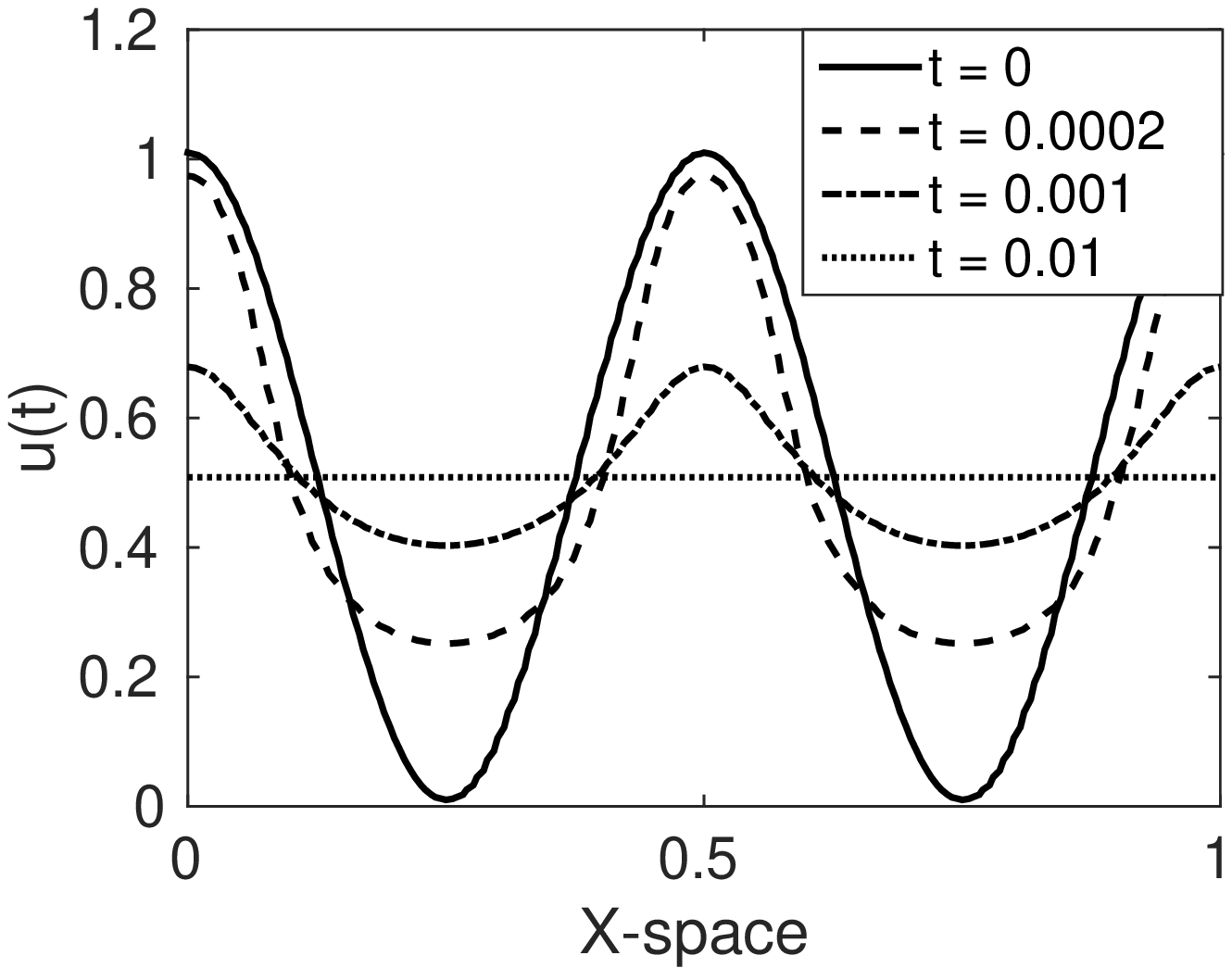}
\includegraphics[width=75mm]{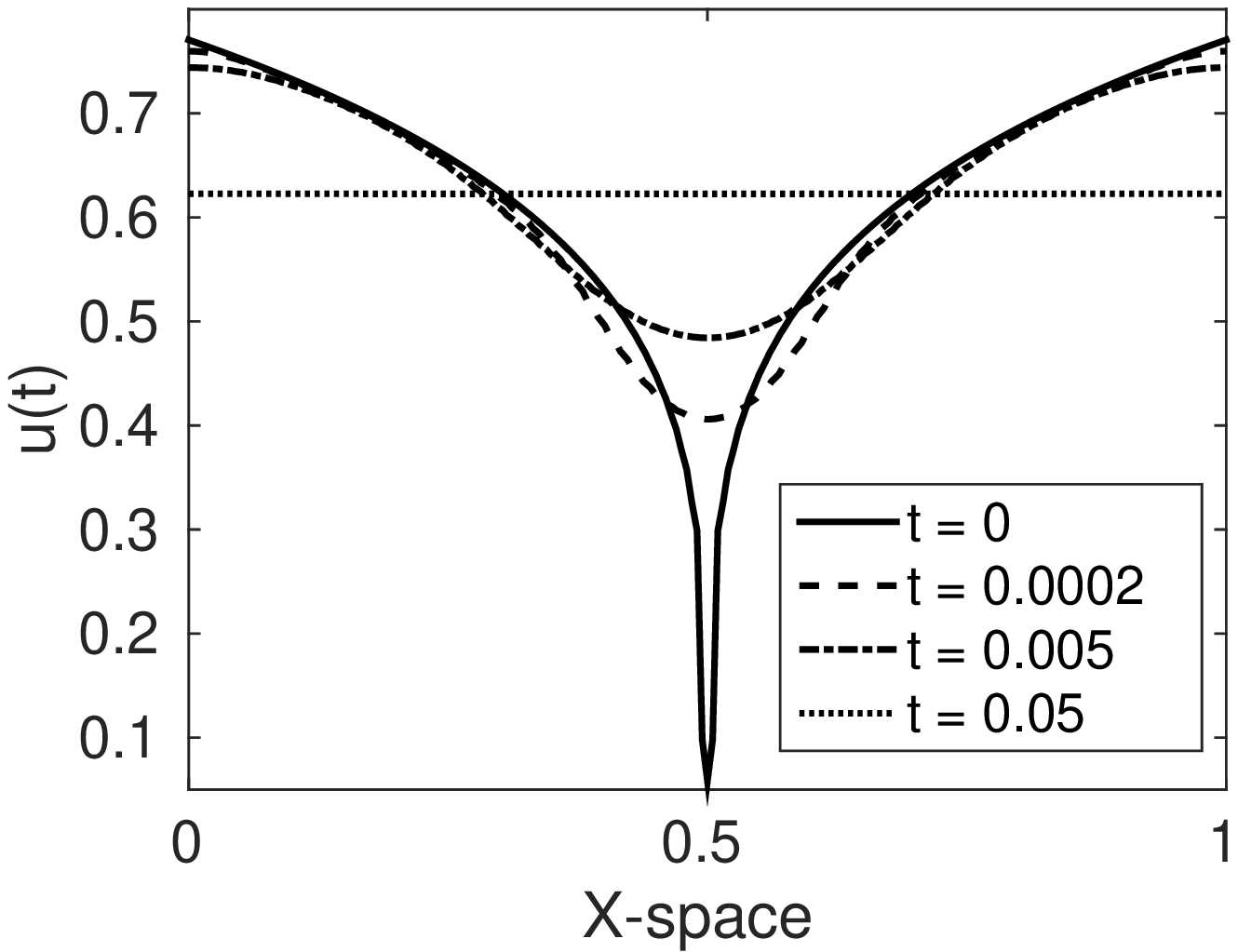}
\caption{Time evolution of the solution to the diffusion equation \eqref{1.eq}
with $\alpha=-1$ for two different initial conditions.}
\label{fig.evol}
\end{figure}

A nice feature of the Wasserstein gradient flow scheme is that we may interpret
the evolution as a process of redistribution of particles with spatio-temporal
density $u(x,t)$ on $\T$ under the influence of a nonlinear particle interaction,
which is described by $S$. The way in which
the initial density $u^0$ is ``deformed'' during the time evolution is illustrated
in Figure \ref{fig.pt}. We have chosen 50 ``test particles'' for the solutions
to \eqref{1.eq} for the initial conditions chosen above. 
%The test particles
%are distributed initially uniformly in space for better visualization.
We stress the fact that the density of trajectories can generally be not
identified with the density $u$ of the solution.

\begin{figure}[ht]
\includegraphics[width=75mm]{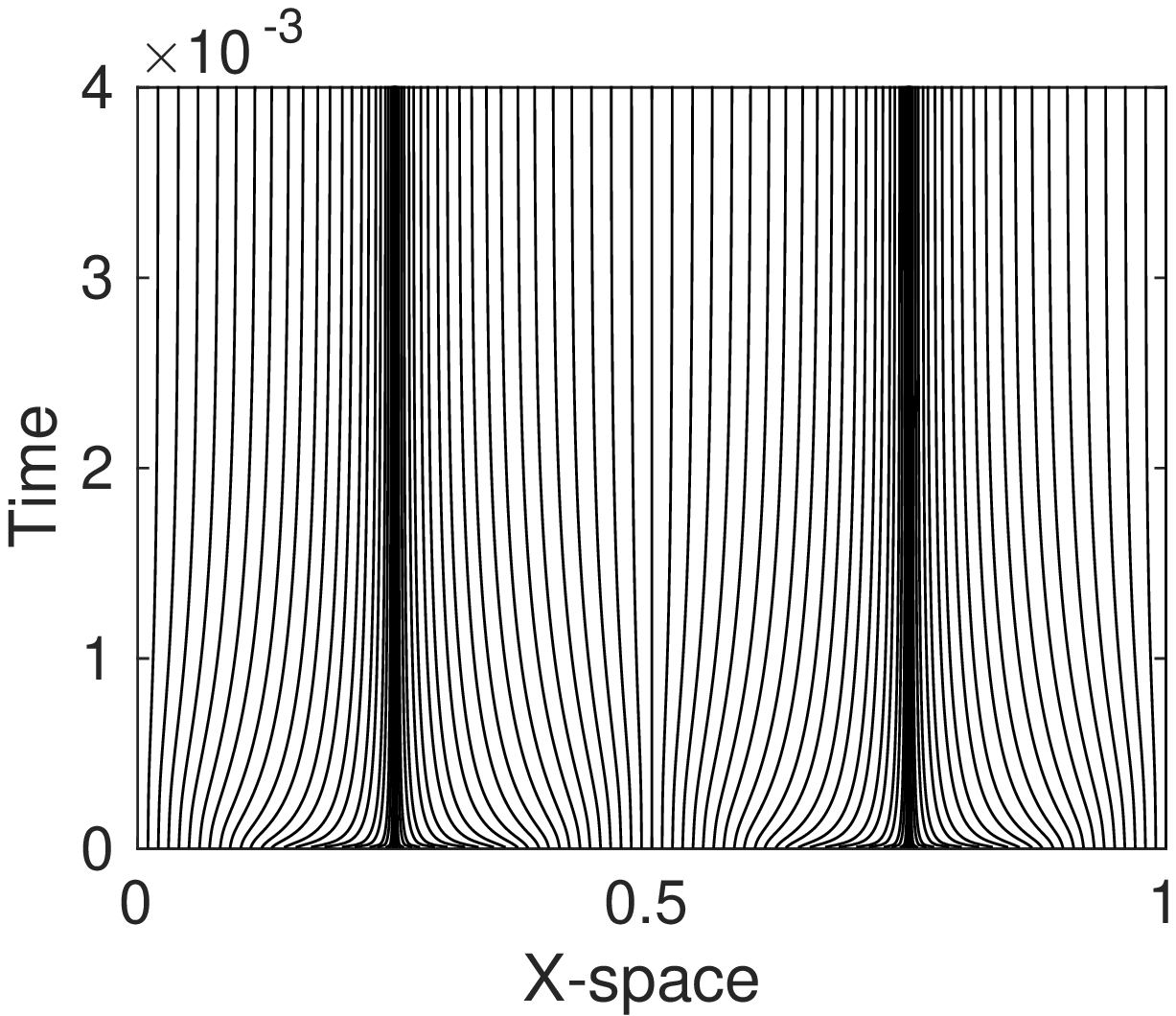}
\includegraphics[width=75mm]{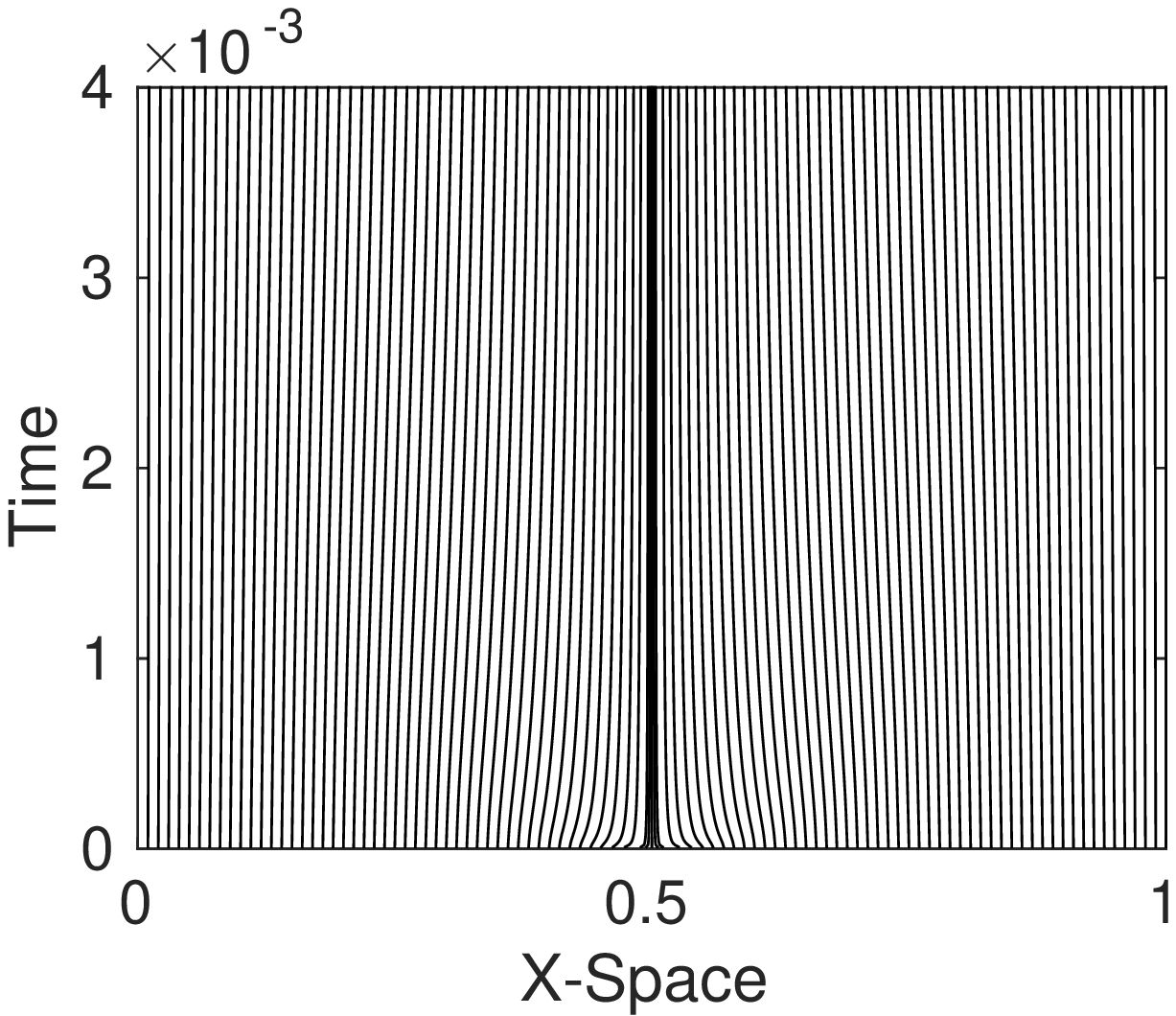}
\caption{Particle trajectories in the Wasserstein gradient flow scheme,
corresponding to the solutions of Figure \ref{fig.evol} with $N=50$.}
\label{fig.pt}
\end{figure}

We verify that the discretization is indeed of second order.
Figure \ref{fig.err1} shows the $\ell^\infty$-error 
for various numbers of grid points $N$.
We have chosen the initial datum $u_0(x)=\cos(2\pi x)^2 + 0.1$, 
the end time $T=0.004$, and the time step size $\tau=10^{-7}$.
The reference solution is computed by using $N=500$, and $\tau=10^{-7}$.
The differences $g(\cdot,T)-g_{\rm ref}(\cdot,T)$ and
$u(\cdot,T)-u_{\rm ref}(\cdot,T)$ in the $\ell^\infty$ norm
feature the expected second-order dependence on $N$.

\begin{figure}[ht]
\includegraphics[width=75mm]{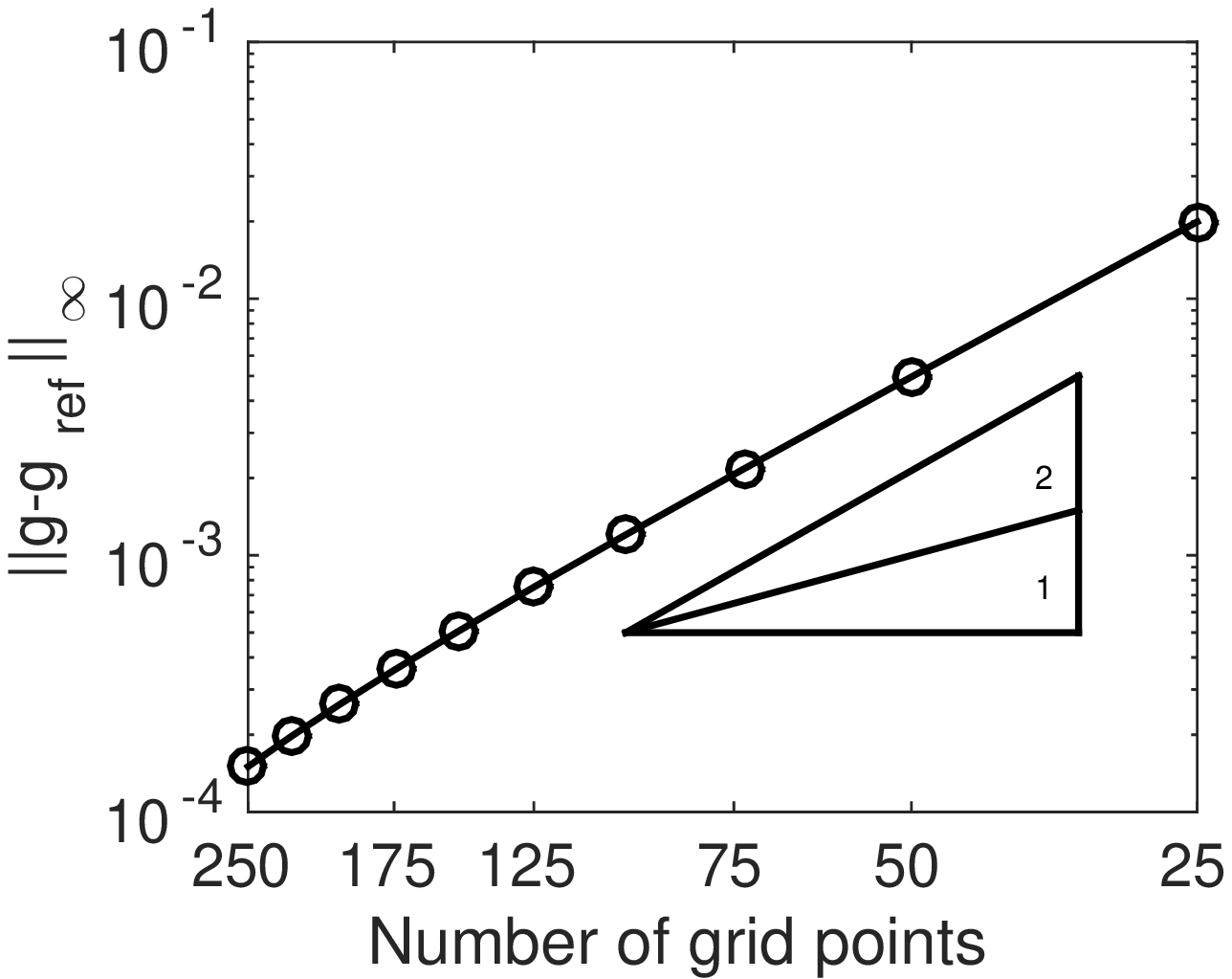}
\includegraphics[width=75mm]{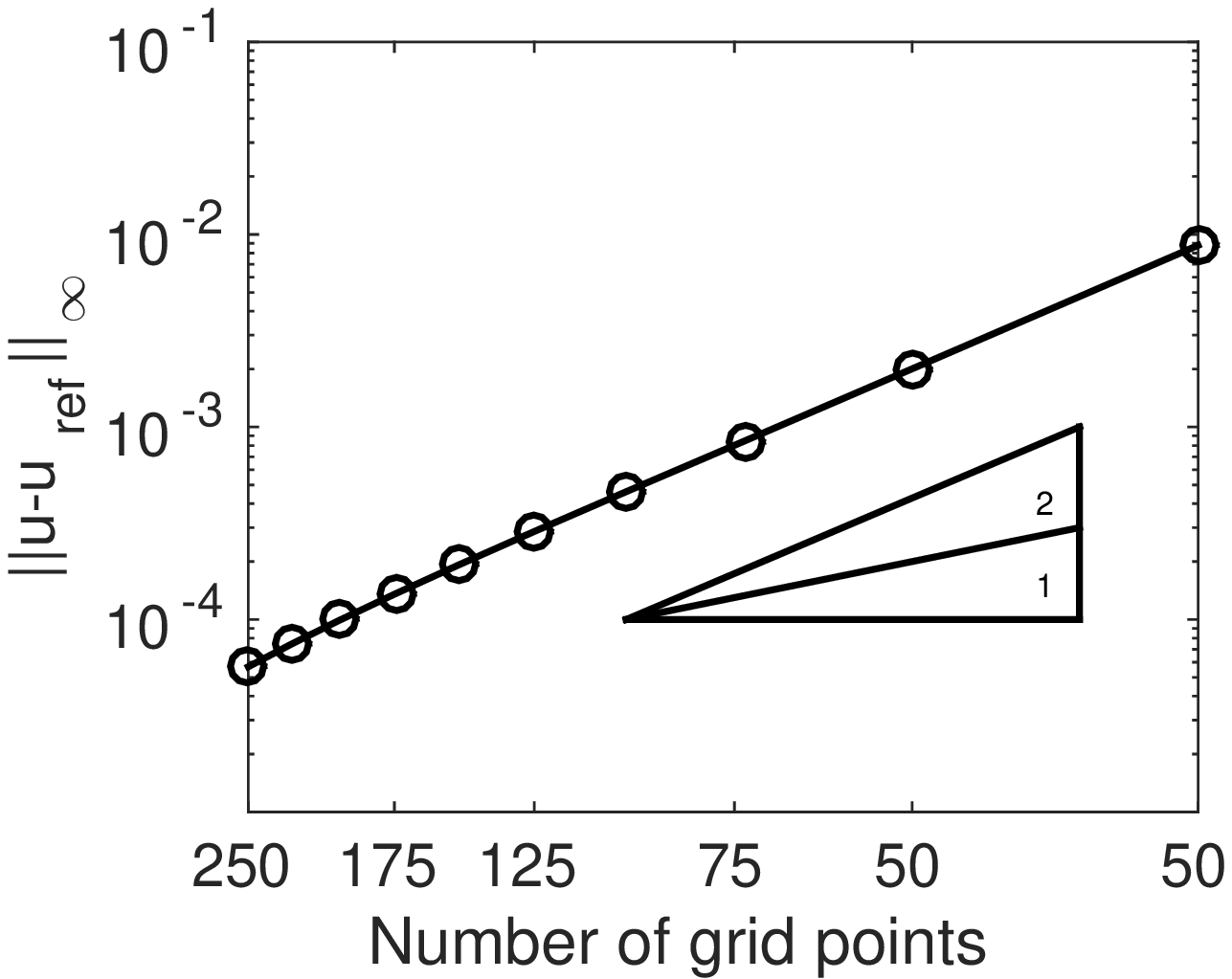}
\caption{$\ell^\infty$-error for $(g-g_{\rm ref})(T)$ (left) and 
$(u-u_{\rm ref})(T)$ (right) at $T=0.004$ for various numbers of grid points.}
\label{fig.err1}
\end{figure}

Next, we fix the number of grid points $N=100$ and compute the 
$L^\infty(\tau^*,T;$ $L^2(\T))$ 
error for varying time step sizes $\tau$; see Figure \ref{fig.err2}. 
Because of the approximation
of the initial datum as detailed in Section \ref{sec.init}, the error will
be not of second order initially. Therefore, we compute the error 
in the interval $(\tau^*,T)$ with $\tau^*=10^{-4}$. 
The errors are of second order, as expected.

\begin{figure}[ht]
\includegraphics[width=75mm]{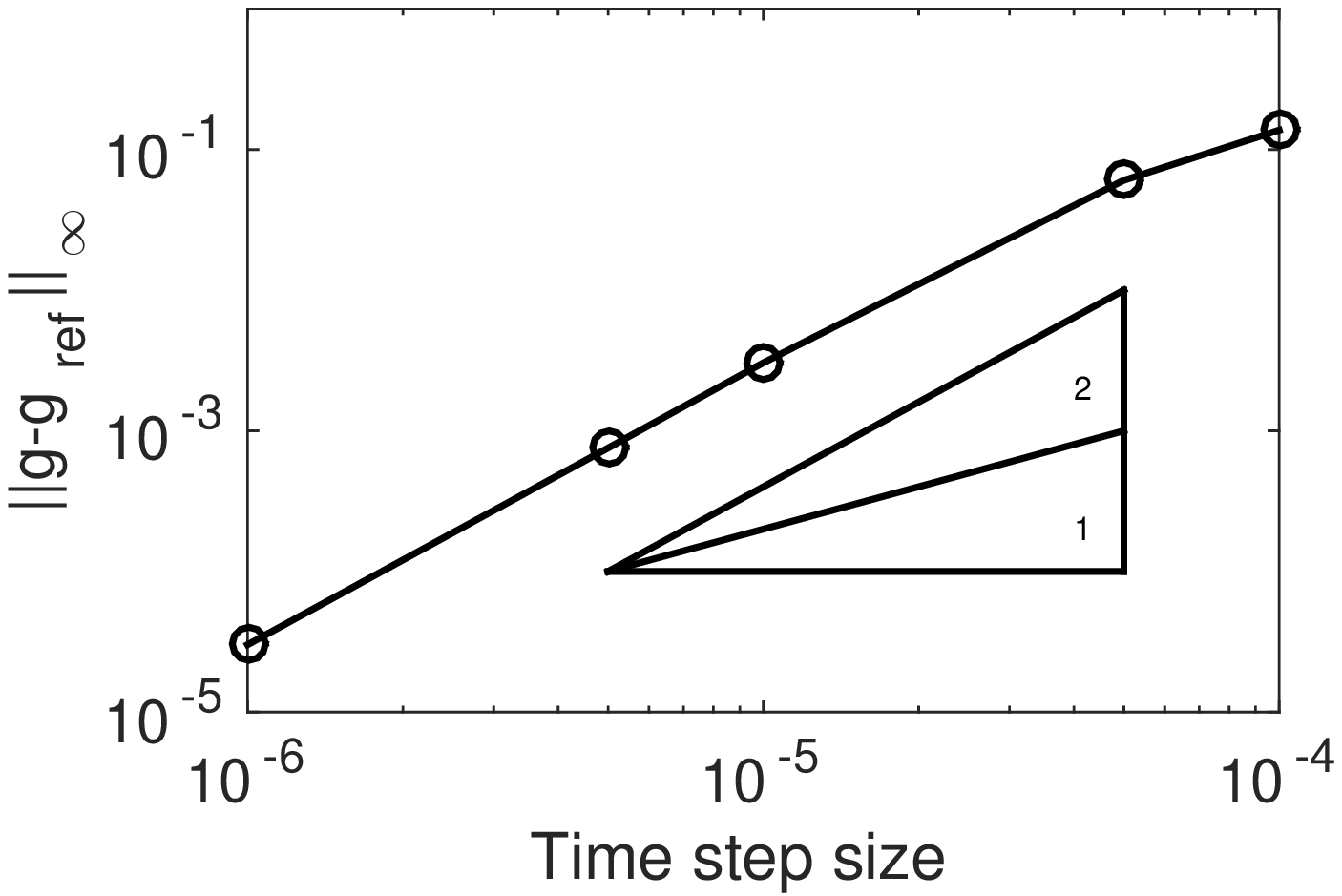}
\includegraphics[width=75mm]{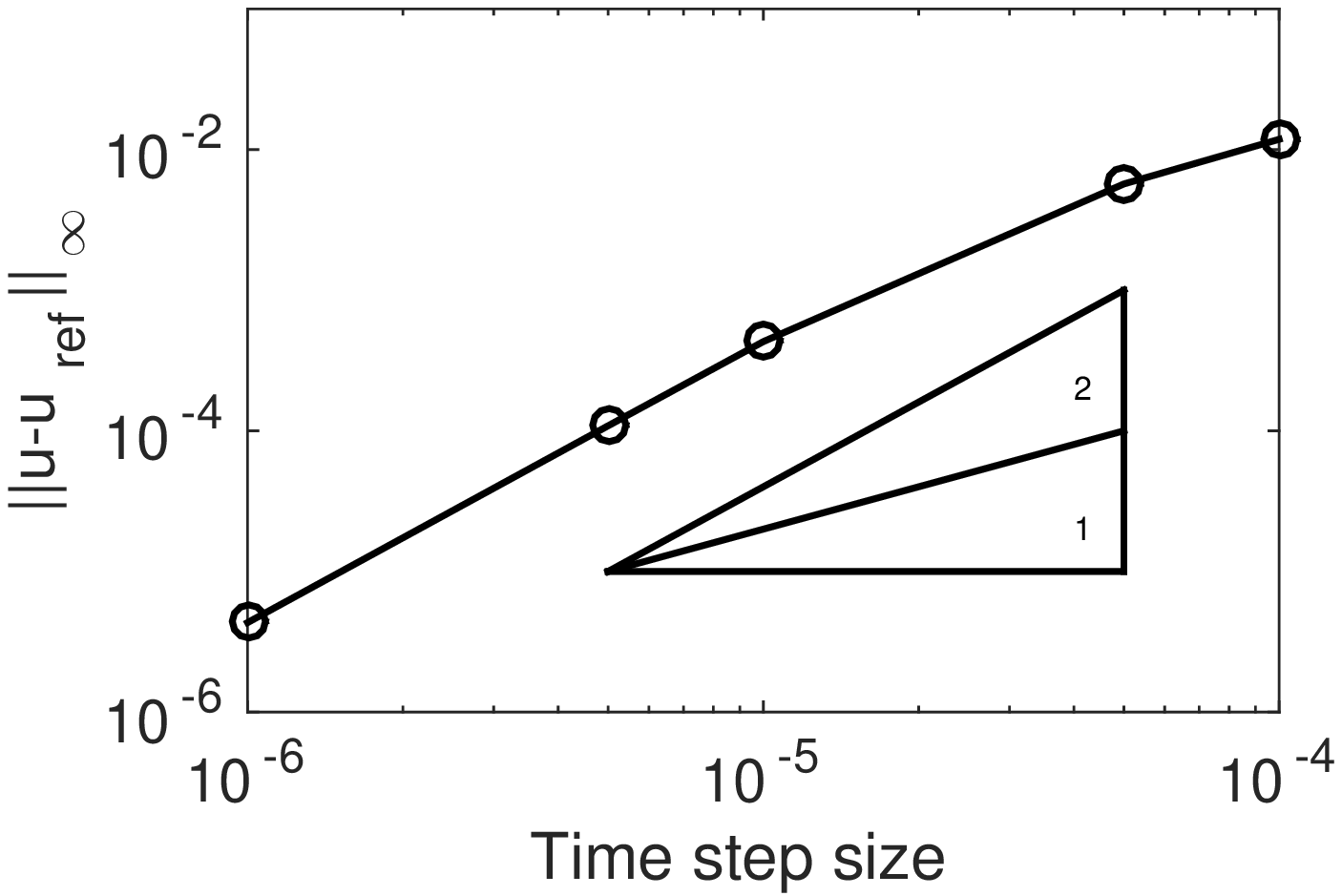}
\caption{$\ell^\infty(\tau^*,T;\ell^2(\T))$-error for $g-g_{\rm ref}$ (left) and 
$u-u_{\rm ref}$ (right) for various time step sizes, with $\tau^*=10^{-4}$.}
\label{fig.err2}
\end{figure}

The time decay of the discrete version of the relative entropy $S[u]-S[u_\infty]$
is presented in Figure \ref{fig.ent} (left) for various grid numbers. We observe
that the decay is exponential until saturation. The
saturation comes from the spatial error and the error from the Newton iteration.
The decay rate is estimated in the linear regime from the difference quotient
$$
  \lambda \approx \frac{1}{\tau}\big(\log S[u(t+\tau)]-\log S[u(t)]\big).
$$
%The values of $\lambda$ for various time step sizes are shown in the right panel of
%Figure \ref{fig.ent}. We observe that the rates increase for smaller time
%step sizes, i.e., the solution converges faster for ``small'' time step sizes.

\begin{figure}[ht]
\centering\includegraphics[width=75mm]{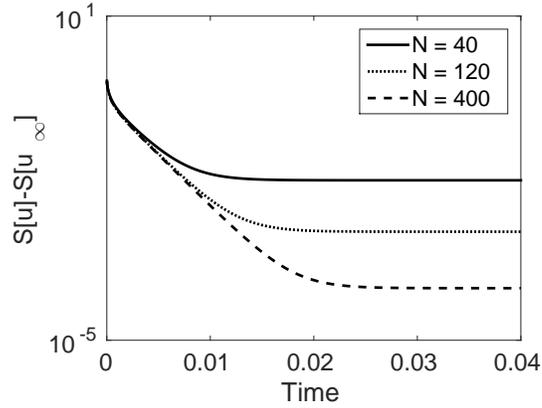}
\caption{Discrete relative entropy $S[u(t)]-S[u_\infty]$ versus time
with $\alpha=-1$ for various $N$. 
%Right: Estimated decay rate versus time step size $\tau$ for $\alpha=-1$.
}
\label{fig.ent}
\end{figure}

The numerical decay rates for $\alpha=-1,-2$ are shown in Figure \ref{fig.ent2}.
The rate for $\alpha=-2$ (right) is much larger than the corresponding one for
$\alpha=-1$ (left), since a smaller exponent yields a larger diffusion
coefficient (if $u<1$) and thus, diffusion becomes faster. 
We also see that the decay rates become larger on a finer spatial grid.
This behavior seems to confirm recent analytical results for spatial discretizations
of Fokker-Planck equations; see \cite[Section 5]{Mie13}. One may ask if a similar 
behavior can be observed for the decay rate as a function of the time step size.
However, our numerical experiments do not show a monotonic dependence
(figures not presented); rather the decay rates vary in a small range 
which seems to be determined by the other numerical error parts.

\begin{figure}[ht]
\includegraphics[width=75mm]{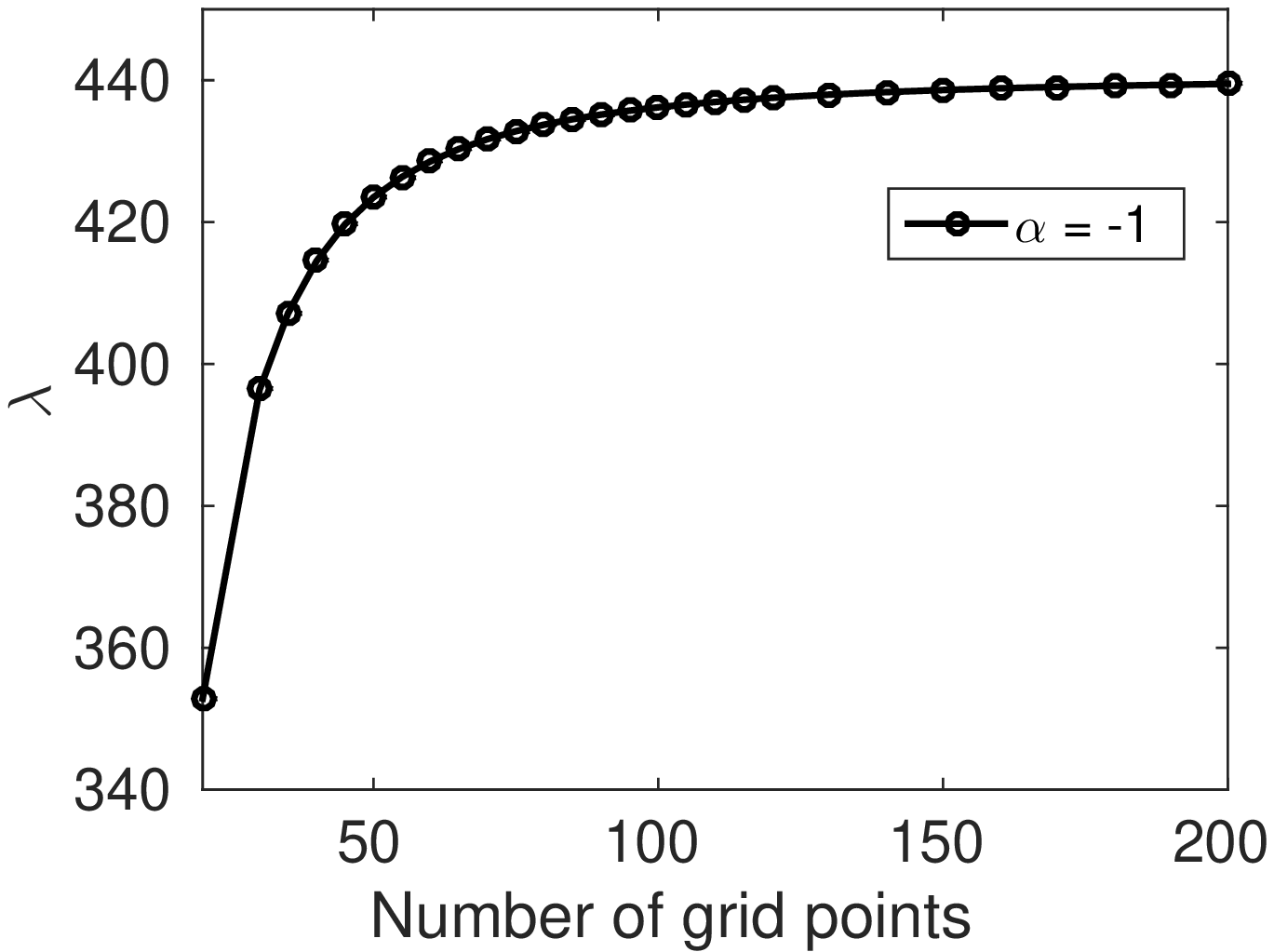}
\includegraphics[width=75mm]{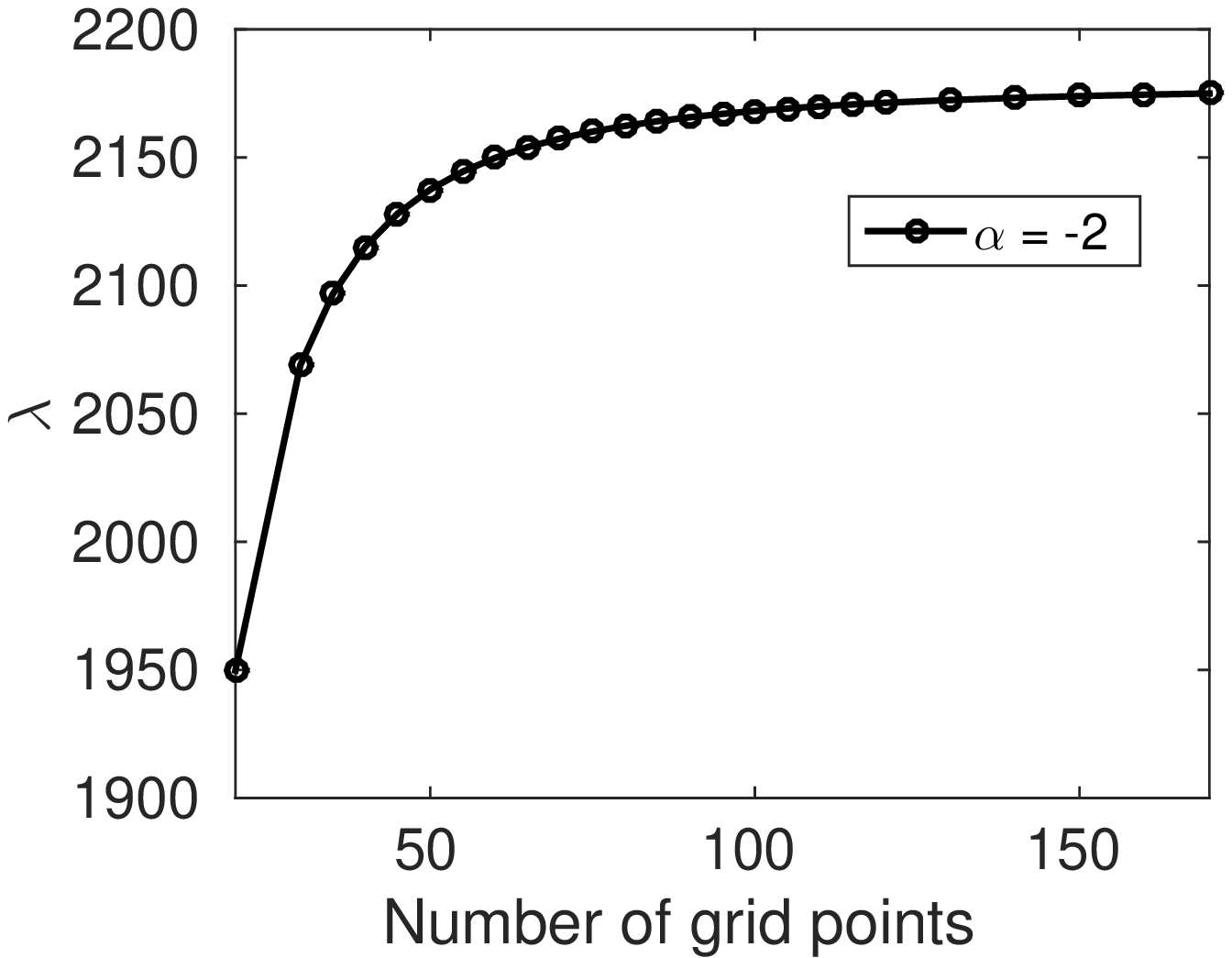}
\caption{Estimated decay rate of the entropy 
versus number of grid points $N$ for $\alpha=-1$ (left) and $\alpha=-2$ (right).}
\label{fig.ent2}
\end{figure}

In Figure \ref{fig.G}, the decay of the square of the relative $G$-norm 
is presented.
The $G$-norm is calculated according to \eqref{1.Gnorm}, where the argument
is given by $\mathbf{g}-\mathbf{g}_\infty$ and $\mathbf{g}_\infty$ is the
weight vector corresponding to the constant steady state. Again, 
the decay is much faster for $\alpha=-2$ because of the faster diffusion.

\begin{figure}[ht]
\includegraphics[width=75mm]{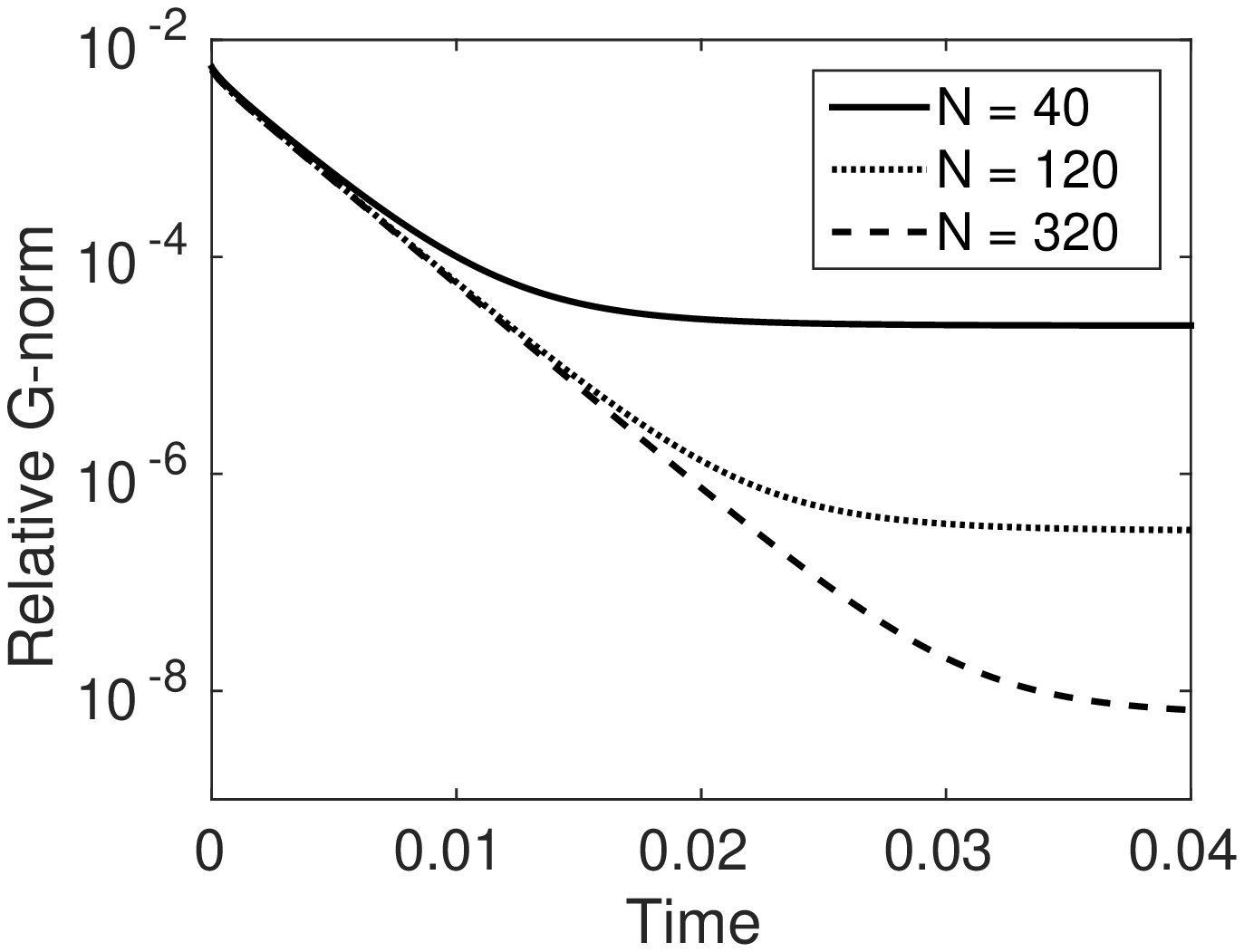}
\includegraphics[width=75mm]{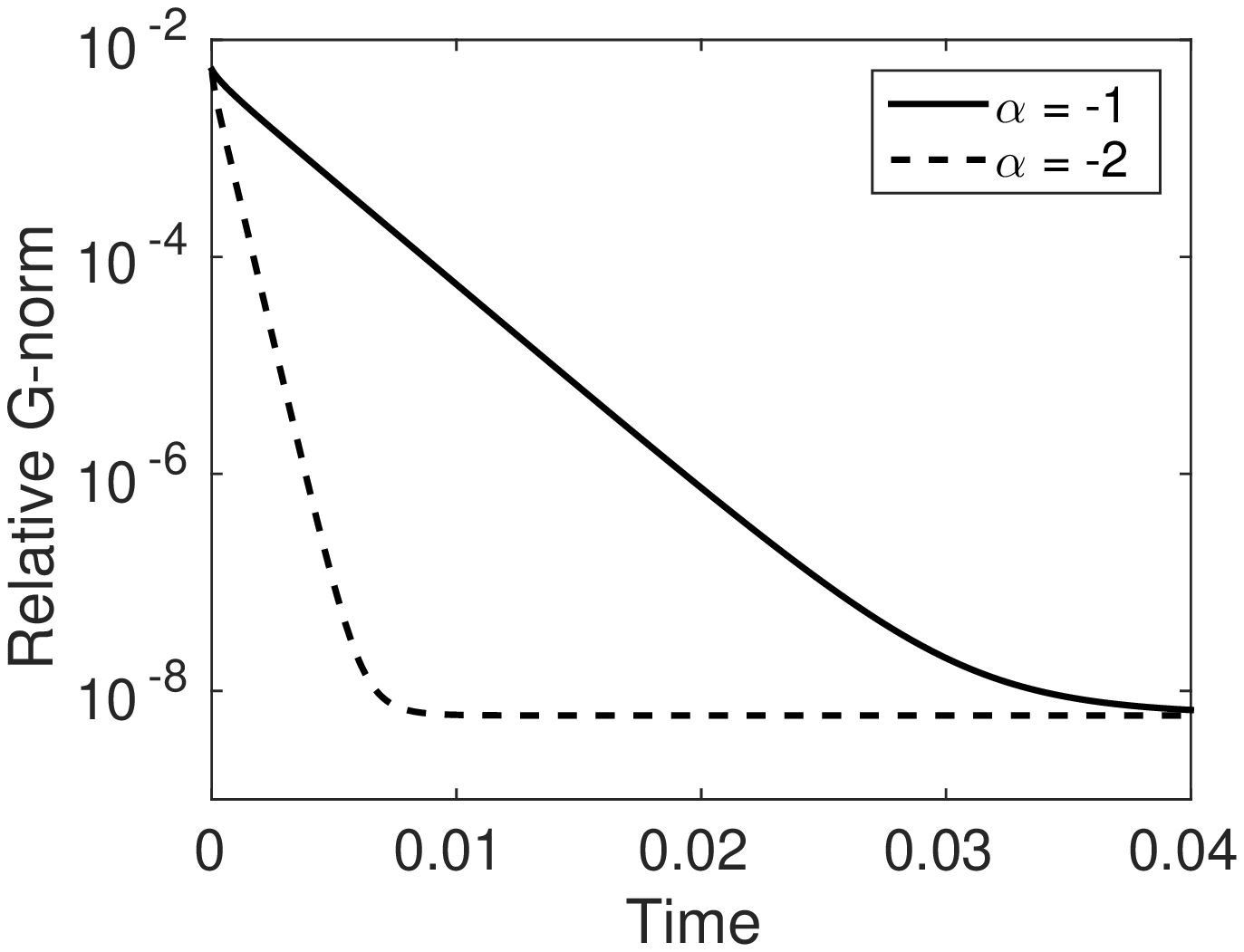}
\caption{Relative $G$-norm versus time for various grid numbers $N$ (left)
and for two values of the exponent $\alpha$ (right).}
\label{fig.G}
\end{figure}

Finally, we present some results on the time decay of the discrete
variance of $u^n$ and $g^n$ at time $\tau n$, defined by 
$$
  \mbox{Var}(u^n)^2 = \sum_{i=1}^{N+1}(u^n_i-E)^2(x_i-x_{i-1}),
$$
where $E$ is the expectation value of $u^n$ (which equals the mass and is
therefore constant in time). The discrete variance of $g^n$ is defined in a 
similar way. Interestingly, the variances are exponentially decaying
(Figure \ref{fig.var}), although it is not clear how to prove this property
analytically.

\begin{figure}[ht]
\includegraphics[width=75mm]{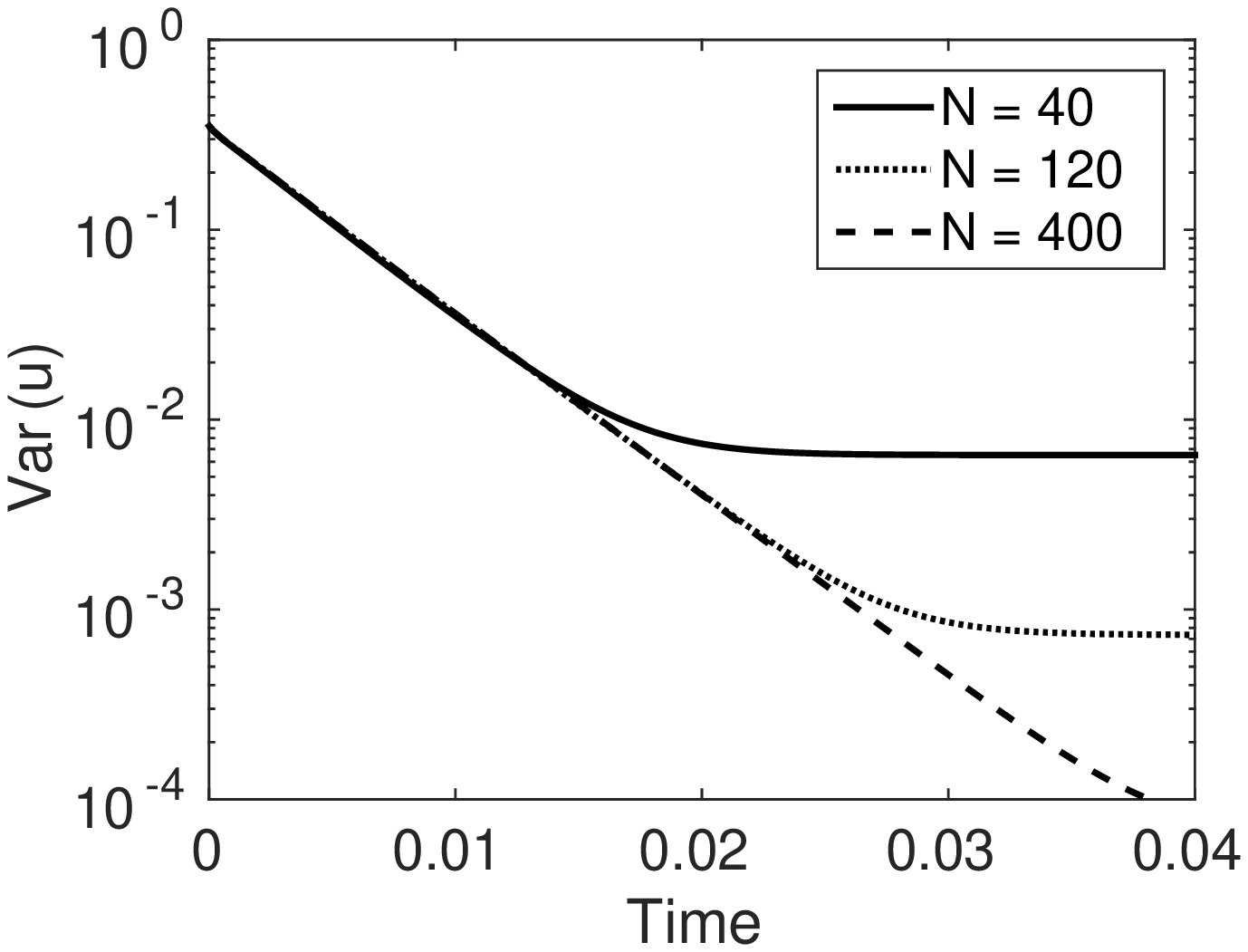}
\includegraphics[width=75mm]{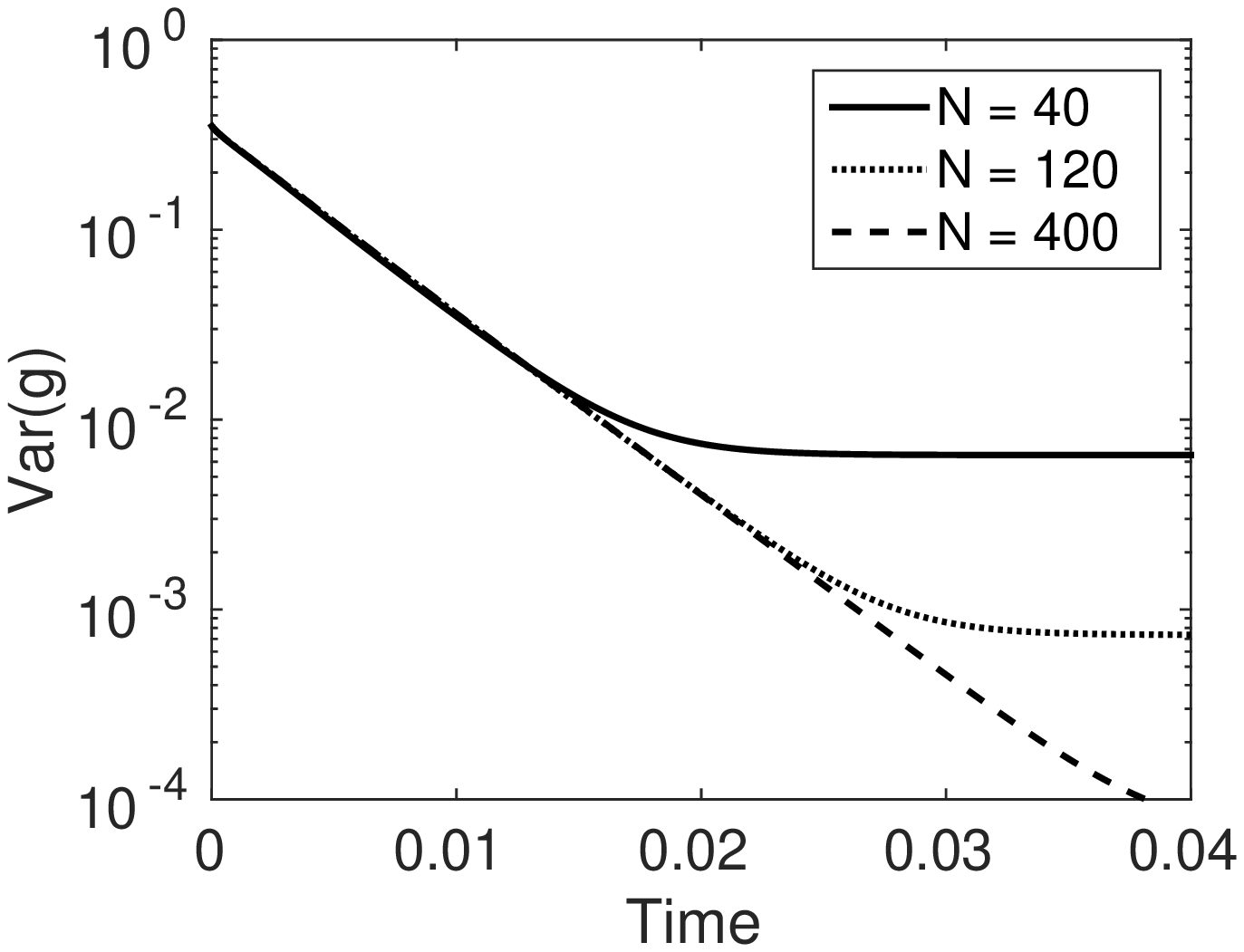}
\caption{Discrete variance of $u^n$ (left) and $g^n$ (right) versus time 
for various grid numbers $N$.}
\label{fig.var}
\end{figure}

%%%%%%%%%%%%%%%%%%%%%%%%%%%%%%%%%%%%%%%%%%%%%%%%%%%%%%%%%%%%%%%%%%%%%%%%%%

\section{Conclusion}
\label{sec.conc}

We have proposed an extension of the standard minimizing movement scheme
to higher-order BDF time discretizations. Quadratic finite elements
have been used to discretise in (mass) space. As an example, we have considered
a (singular) diffusion equation, but it should be possible to adapt the scheme 
to other partial differential equations which constitute $L^2$ Wasserstein
gradient flows. The implicit numerical scheme is based on successive solution
of constrained minimization problems. The quadratic convergence in space and time
of our method has been confirmed numerically. It turns out that the relative
entropy, $G$-norm, and variance converge exponentially to zero. An interesting
observation is that the decay rate depends monotonically on the number 
of grid numbers, but it seems that there is no monotonic behavior with respect to 
the time step size. Future work may be concerned with an analytical derivation of
numerical decay rates using a discrete variant of the Bakry-Emery approach.
First steps in this direction were presented in \cite{Mie13} but only for
a semi-discretization. Possibly this approach has to be adapted to $G$-stable
schemes in the spirit of \cite{JuMi15}.

%%%%%%%%%%%%%%%%%%%%%%%%%%%%%%%%%%%%%%%%%%%%%%%%%%%%%%%%%%%%%%%%%%%%%%%%%%

\begin{appendix}
\section{Proofs of Theorems \ref{thm.ex} and \ref{thm.time}}\label{app.ex}

\begin{proof}[Proof of Theorem \ref{thm.ex}]
The existence proof is based on a regularization of the initial datum and
a fixed-point argument.
Let $0<\eps<1$ and $u_{0,\eps}=u^0+\eps$. Let $Q_T=\T^d\times(0,T)$
and $M=\sup_{\T^d} u^0$, Set 
$$
  K=\{u\in L^2(Q_T):\eps\le u\le M,\ \|u\|_{L^2(0,T;H^1(\T^d))}
	+\|\pa_t u\|_{L^2(0,T;H^1(\T^d)')}\le C\},
$$
where $C>0$ will be determined later. The set $K$ is convex and, by Aubin's lemma, 
compact in $L^2(Q_T)$. Let $v\in K$ and let $u\in L^2(Q_T)$ be the weak
solution to
\begin{equation}\label{app.1}
  \pa_t u = \diver(v^{\alpha-1}\na u)\quad\mbox{in }\T^d,\ t>0, \quad u(0)=u_{0,\eps}.
\end{equation}
This defines the fixed-point operator $Z:K\to L^2(Q_T)$, $v\mapsto u$. Standard
arguments show that $Z$ is continuous. We verify that $Z(K)\subset K$.
By the maximum principle, $\eps\le u\le M$. Using $u$ as a test function in the
weak formulation of \eqref{app.1} shows that 
$\|u\|_{L^2(0,T;H^1(\T^d))}\le C_1(\eps)$, where $C_1(\eps)>0$ is some constant
depending on $\eps$. Moreover, $\|\pa_t u\|_{L^2(0,T;H^1(\T^d)')}
\le \|v^{\alpha-1}\na u\|_{L^2(Q_T)}\le C_2(\eps)$. 
Thus, setting, $C:=C_1(\eps)+C_2(\eps)$, we infer that $u\in K$.
By the fixed-point theorem of Schauder, there exists a fixed point $u_\eps$ of $Z$.

In order to perform the limit $\eps\to 0$, we need to derive $\eps$-independent
estimates for $u_\eps$. To this end, we need to distinguish several cases.
First, let $\alpha=-1$. Employing the test function $1-1/u_\eps$
in the weak formulation of \eqref{app.1} with $u=v=u_\eps$, we find that
$$
  \int_{\T^d}(u_\eps(t)-\log u_\eps(t))\, dx 
	+ \int_0^t\int_{\T^d}|\na u_\eps^{-1}|^2 \, dxds
	= \int_{\T^d}(u^0+\eps-\log(u^0+\eps))\, dx.
$$
The right-hand side is uniformly bounded as we assumed that
$-\int_{\T^d}\log u^0 \, dx<\infty$.
Since $|\na u_\eps^{-1}|^2 \ge M^{-4}|\na u_\eps|^2$, we infer uniform
estimates for $u_\eps$ in $L^2(0,T;H^1(\T^d))$ and also in $H^1(0,T;H^1(\T^d)')$.

Next, let $\alpha\neq -1$. The test function $u_\eps^\alpha$ 
in the weak formulation of \eqref{app.1} gives
$$
  \frac{1}{\alpha+1}\int_{\T^d}u_\eps(t)^{\alpha+1}\, dx
	+ \frac{1}{\alpha}\int_0^t\int_{\T^d}|\na u_\eps^\alpha|^2\, dxds 
	= \frac{1}{\alpha+1}\int_{\T^d}(u^0+\eps)^{\alpha+1}\, dx,
$$
If $-1<\alpha<0$, we write this equation as
\begin{align*}
  \int_0^t\int_{\T^d}|\na u_\eps^\alpha|^2\, dxds 
	&= -\frac{\alpha}{\alpha+1}\int_{\T^d}u_\eps(t)^{\alpha+1}\, dx
	+ \frac{\alpha}{\alpha+1}\int_{\T^d}(u^0+\eps)^{\alpha+1}\, dx \\
	&\le -\frac{\alpha}{\alpha+1}\int_{\T^d}M^{\alpha+1}\, dx.
\end{align*}
If $\alpha<-1$, we obtain
$$
  \frac{1}{-\alpha-1}\int_{\T^d}u_\eps(t)^{\alpha+1}\, dx
	+ \frac{1}{-\alpha}\int_0^t\int_{\T^d}|\na u_\eps^\alpha|^2\, dxds 
	= \frac{1}{-\alpha-1}\int_{\T^d}(u^0+\eps)^{\alpha+1}\, dx.
$$
In both cases, since $u^0$ is assumed to be bounded, we infer a uniform
bound for $u_\eps^\alpha$ in $L^2(0,T;H^1(\T^d))$ and consequently also 
for $\pa_t u_\eps$ in $L^2(0,T;H^1(\T^d)')$. Moreover, in view of
$|\na u_\eps^\alpha|^2 = \alpha^2 u_\eps^{2(\alpha-1)}|\na u_\eps|^2
\ge \alpha^2 M^{2(\alpha-1)}|\na u_\eps|^2$, it follows that $(u_\eps)$ is
bounded in $L^2(0,T;H^1(\T^d))$.

We infer for all $\alpha<0$ the following bounds:
$$
  \|u_\eps\|_{L^\infty(0,T;L^\infty(\T^d))} + \|u_\eps^\alpha\|_{L^2(0,T;H^1(\T^d))}
	+ \|u_\eps\|_{L^2(0,T;H^1(\T^d))} + \|\pa_t u_\eps\|_{L^2(0,T;H^1(\T^d)')} \le C_3.
$$
By Aubin's lemma, there exists a subsequence which is not relabeled such that 
$u_\eps\to u$ strongly in $L^2(Q_T)$ as $\eps\to 0$.
Moreover, $u_\eps^\alpha\rightharpoonup u^\alpha$ weakly in $L^2(0,T;H^1(\T^d))$
and $\pa_t u_\eps\rightharpoonup \pa_t u$ weakly in $L^2(0,T;H^1(\T^d)')$.
Thus, we may pass to the limit $\eps\to 0$ in the weak formulation which shows
that $u$ solves \eqref{1.eq}.
\end{proof}

\begin{proof}[Proof of Theorem \ref{thm.time}]
%Let $u$ be a positive solution to \eqref{1.eq}.
%Let $1<\beta\le 2$ and $\mbox{vol}(\T^d)=1$. 
Employing \eqref{1.eq} and integration by parts, we find that
$$
  \frac{dH_\beta}{dt} = -\beta(\beta-1)\int_{\T^d}u^{\alpha+\beta-3}|\na u|^2 \, dx
	= -\frac{4}{\beta}(\beta-1)\int_{\T^d}u^{\alpha-1}|\na u^{\beta/2}|^2 \, dx.
$$
We employ the bound $u\le M=\sup_{\T^d}u^0$ and the Beckner inequality \cite{Bec89}
(note that we assumed that $\mbox{vol}(\T^d)=1$),
\begin{equation}\label{app.bec}
  \int_{\T^d}u^\beta \, dx - \bigg(\int_{\T^d}u\, dx\bigg)^\beta
	\le C_B\int_{\T^d}|\na u^{\beta/2}|^2 \, dx \quad\mbox{for }
	u^{\beta/2}\in H^1(\Omega),\ 1<\beta\le 2,
\end{equation}
to obtain
$$
  \frac{dH_\beta}{dt} \le -\frac{4(\beta-1)}{\beta C_B}M^{\alpha-1}
	\bigg(\int_{\T^d}u^\beta \, dx - \bigg(\int_{\T^d}u\, dx\bigg)^\beta\bigg)
	= -\frac{4(\beta-1)}{\beta C_B}M^{\alpha-1}H_\beta.
$$
Then Gronwall's lemma yields $H[u(t)]\le H[u^0]e^{-\lambda t}$ with
$\lambda=-4(\beta-1)M^{\alpha-1}/(\beta C_B)$ for $1<\beta\le 2$. 

For the second result, let $-1\le\alpha<0$ and $\beta=2(1-\alpha)$. 
Similarly as above, we find that
\begin{align*}
  \frac{dH_\beta}{dt} 
	&= -\frac{4\beta(\beta-1)}{(\alpha+\beta-1)^2}
	\int_{\T^d}|\na u^{(\alpha+\beta-1)/2}|^2 \, dx \\
	&\le -\frac{8(1-2\alpha)}{1-\alpha}\int_{\T^d}\bigg(\int_{\T^d}
	u^{\alpha+\beta-1}\, dx - \bigg(\int_{\T^d}u\, dx\bigg)^{\alpha+\beta-1}\bigg) \\
	&= -\frac{8(1-2\alpha)}{1-\alpha}H_{\beta/2}[u],
\end{align*}
since $\alpha+\beta-1=\beta/2\in(1,2]$. Using the inequalities
$\|u^{\beta/2}\|_{L^2(\T^d)}\ge\|u^{\beta/2}\|_{L^1(\T^d)}$ and
$\|u^{\beta/2}\|_{L^2(\T^d)}=\|u\|_{L^\beta(\T^d)}^{\beta/2}
\ge \|u\|_{L^1(\T^d)}^{\beta/2}$
(again we employ $\mbox{vol}(\T^d)=1$ here), it follows that
\begin{align*}
  H_\beta[u] &= \|u^{\beta/2}\|_{L^{2}(\T^d)}^2 - \|u\|_{L^1(\T^d)}^\beta \\
	&= \big(\|u^{\beta/2}\|_{L^{2}(\T^d)}+\|u\|_{L^1(\T^d)}^{\beta/2}\big)
	\big(\|u^{\beta/2}\|_{L^{2}(\T^d)}-\|u\|_{L^1(\T^d)}^{\beta/2}\big) \\
	&\ge \big(\|u^{\beta/2}\|_{L^{2}(\T^d)}+\|u\|_{L^1(\T^d)}^{\beta/2}\big)
	\big(\|u^{\beta/2}\|_{L^{1}(\T^d)}-\|u\|_{L^1(\T^d)}^{\beta/2}\big) \\
	&\ge \big(\|u^{\beta/2}\|_{L^{1}(\T^d)}+\|u\|_{L^1(\T^d)}^{\beta/2}\big)
	H_{\beta/2}[u] \ge 2\|u\|_{L^1(\T^d)}^{\beta/2}H_{\beta/2}[u].
\end{align*}
Since the solution to \eqref{1.eq} conserves mass, $\|u(t)\|_{L^1(\T^d)}
=\|u^0\|_{L^1(\T^d)}$, and we end up with
$$
  \frac{dH_\beta}{dt} \le -\frac{4(1-2\alpha)}{(1-\alpha)
	\|u^0\|_{L^1(\T^d)}^{\beta/2}}H_\beta[u].
$$
Then Gronwall's lemma shows that $H_\beta[u(t)]\le H_\beta[u^0]e^{-\lambda t}$ with
$\lambda=4(1-2\alpha)/((1-\alpha)\|u^0\|_{L^1(\T^d)}^{1-\alpha})$.

Finally, the statement of the theorem follows after applying the generalized
Csisz\'ar-Kullback inequality in the form
$$
  \|u-v\|_{L^1(\T^d)}^2 \le C_\beta\|v\|_{L^1(\T^d)}\bigg(
	\int_{\T^d} u^\beta - \bigg(\int_{\T^d}u\, dx\bigg)^\beta\bigg), \quad 1<\beta\le 2.
$$
for functions $u$, $v\in L^\beta(\T^d)$ such that $\int_{\T^d}u\, dx=\int_{\T^d}v\, dx$.
The proof is a slight generalization of the proof of Theorem 1.4 in
\cite{GoLe10} taking $\varphi(t)=t^\beta$.
\end{proof}

%%%%%%%%%%%%%%%%%%%%%%%%%%%%%%%%%%%%%%%%%%%%%%%%%%%%%%%%%%%%%%%%%%%%%%%%%$

\section{Computations}

In this appendix, we detail the calculations for the coefficients of the
matrix \eqref{disc.abc}, and the Hessian of the discrete
entropy \eqref{disc.SN}, both in the case $\alpha=-1$.

\subsection{Computation of the coefficients $M_{ij}$}\label{sec.coeff}

We compute the coefficients of the matrix \eqref{disc.M}, i.e.\ the 
coefficients $a_{ij}$, $b_{ij}$, and $c_{ij}$ defined in \eqref{disc.abc}.
In the following, we set
$$
  \delta_j = \omega_j-\omega_{j-1}, \quad
	\Delta_j = \frac12(\omega_{j+1}-\omega_{j-1}), \quad
	\sigma_j = \frac13(\omega_{j+1}+\omega_j+\omega_{j-1}).
$$

\begin{lemma}[Coefficients $a_{ij}$]\label{lem.a}
The coefficients of the symmetric matrix $A=(a_{ij})$, defined in \eqref{disc.abc},
read as
\begin{align*}
  a_{jj} &= \Delta_j^2(M-\sigma_j) - \frac{\Delta_j}{60}(12\Delta_j^1
	+ \delta_j^2 + \delta_{j+1}^2), \quad 1\le j\le N-1, \\
	a_{j,j+1} &= \Delta_j\Delta_{j+1}(M-\sigma_{j+1}) - \frac{\delta_j^3}{120},
	\quad 1\le j\le N-1, \\
	a_{jk} &= \Delta_j\Delta_k(M-\sigma_k), \quad j+2\le k\le N-1, \\
	a_{1N} &= \frac12\Delta_1\Delta_N\bigg(M-\frac{\omega_2}{3}\bigg)
	- \frac{\Delta_N^3}{120}, \\
	a_{jN} &= \frac12\Delta_j\Delta_N\bigg(M-\sigma_j+\frac{\delta_N}{3}\bigg),
	\quad 2\le j\le N-2, \\
	a_{N-1,N} &= \frac12\Delta_{N-1}\Delta_N\bigg(M-\frac13(\omega_{N-2}+2\omega_{N-1})
	\bigg)-\frac{\Delta_N^3}{120}, \\
	a_{NN} &= \frac{M}{4}\Delta_N^2 + \frac{\Delta_N^3}{10}.
\end{align*}
\end{lemma}

This lemma was proved in \cite[Lemma 2.6]{DMM10}. For the convenience of the reader,
we recall the proof below.

\begin{proof}
We reformulate the integral in \eqref{disc.abc}:
\begin{align*}
  & a_{jk} = M\int_0^M\phi_j(\eta)\,d\eta\int_0^M\phi_k(\eta')\,d\eta' 
	- \int_0^M\phi_j(\eta)\,d\eta\int_0^M\eta'\phi_k(\eta')\,d\eta'
	- J_{jk}, \\
  &\mbox{where }J_{jk} = \int_0^M\int_0^M(\eta-\eta')_+\phi_j(\eta)
	\phi_k(\eta')\,d\eta d\eta',
\end{align*}
where $(\eta-\eta')_+:=\max\{0,\eta-\eta'\}$. 
The first two integrals become
$$
  \int_0^m\phi_j(\eta)\,d\eta = \delta_j, \quad 
	\int_0^M\eta'\phi_k(\eta')\,d\eta'= \Delta_k\sigma_k.
$$
Since $A$ is symmetric, it is sufficient to consider $1\le j\le k$.
If $j+2\le k<N$, the support of $\phi_j(\eta)\phi_k(\eta')$ is contained in 
$[\omega_{j-1},\omega_{j+1}]\times[\omega_{k-1}\times\omega_{k+1}]$.
Hence, the support is nonvanishing if $\eta\le\omega_{j+1}\le
\omega_{k-1}\le\eta'$, but then $(\eta-\eta')_+=0$
except for $\eta=\eta'$. We conclude that $J_{jk}=0$ and it is sufficient to
compute only $J_{jj}$ and $J_{j,j+1}$:
\begin{align*}
  J_{jj} &= 
	\int_{\omega_{j-1}}^{\omega_{j+1}}\phi_j(\eta)\bigg(\int_{\omega_{j-1}}^\eta
	(\eta-\eta')\phi_j(\eta')\,d\eta'\bigg)\,d\eta 
	= \frac{\Delta_j}{60}(12\Delta_j^2 + \delta_j^2 + \delta_{j+1}^2), \\
	J_{j,j+1} &= \int_{\omega_{j-1}}^{\omega_{j+1}}\phi_j(\eta)\bigg(
	\int_{\omega_{j-1}}^{\max\{\eta,\omega_j\}}
	(\eta-\eta')\phi_j(\eta')\,d\eta'\bigg)\,d\eta
	= \frac{\delta_j^3}{120}.
\end{align*}

Next, let $k=N$. Then the support of $\phi_N$ is contained in 
$[0,\omega_1]\cup[\omega_{N-1},M]$, and we compute:
\begin{align*}
  a_{jN} &= M\int_0^M\phi_j(\eta)\,d\eta\int_0^M\phi_N(\eta')\,d\eta'
	- \int_0^M\eta\phi_j(\eta)\,d\eta\int_0^{\omega_1}\phi_N(\eta')\,d\eta' \\
	&\phantom{xx}{}
	- \int_0^M\phi_j(\eta)\,d\eta\int_{\omega_{N-1}}^M\eta'\phi_N(\eta')\,d\eta'
	- K_j^+ - K_j^- \\
	&= \frac12\Delta_j\Delta_N\bigg(M-\sigma_j+\frac{\delta_N}{3}\bigg)
	- K_j^+ - K_j^-,
\end{align*}
where
\begin{align*}
  K_j^+ &:= \int_0^M\int_\eta^{\omega_1}(\eta'-\eta)_+\phi_j(\eta)\phi_N(\eta')
	\,d\eta d\eta', \\
	K_j^- &:= \int_0^M\int_{\omega_{N-1}}^M(\eta-\eta')_+\phi_j(\eta)\phi_N(\eta')
	\,d\eta d\eta'.
\end{align*}
For $2\le j\le N-2$, the supports of $\phi_j$ and $\phi_N$ do not intersect
such that $K_j^\pm=0$. For $j=1$, we have $K_1^-=0$ and $K_1^+=\Delta_N^3/120$,
whereas for $j=N-1$, $K_{N-1}^+=0$ and $K_{N-1}^-=\Delta_N^3/120$. Furthermore,
$K_N^\pm=\Delta_N^3/30$. Moreover, since $\delta_N=\omega_N-\omega_{N-1}
=(M-\omega_0)-(M-\omega_1)=\omega_1$,
$$
  M-\sigma_1+\frac{\delta_N}{3} = M-\frac13(\omega_0+\omega_1+\omega_2)
	+ \frac13\omega_1 = M-\frac{\omega_2}{3}.
$$
Collecting these results, the lemma follows.
\end{proof}

\begin{lemma}[Coefficients $b_{ij}$]\label{lem.b}
The coefficients of the matrix $B=(b_{ij})$, defined in \eqref{disc.abc}, read as
\begin{align*}
  b_{jj} &= \frac23\delta_j(M\Delta_j - \delta_j\sigma_j) + \beta_{jj}, 
	\quad 1\le j\le N, \\
	b_{j+1,j} &= \frac13\big(2M\delta_{j+1}\Delta_j - (\omega_{j+2}^2-\omega_{j+1}^2)
	\Delta_j\big) + \beta_{j+1,j}, \quad 1\le j\le N-1, \\
	b_{jk} &= \frac23(M\delta_j\Delta_k - \delta_j\delta_k\sigma_k), \quad
	1\le j<k\le N, \\
	b_{jk} &= \frac13\big(2M\delta_j\Delta_k - (\omega_{j+1}^2-\omega_j^2)\Delta_k\big),
  \quad j\ge k-2, \\
	b_{1N} &= \frac23\delta_1\Delta_N - \frac16(\omega_2^2-\omega_1^2)\delta_N
	- \frac19(2\omega_{N}^2 - \omega_{N-1}^2 - \omega_{N-1}\omega_{N})\delta_1 
	- \beta_{1N}, \\
	b_{jN} &= \frac23\delta_j\Delta_N - \frac16(\omega_{j+1}^2-\omega_j^2)\delta_N
	- \frac19(2\omega_{N}^2 - \omega_{N-1}^2 - \omega_{N-1}\omega_{N})\delta_j,
	\quad 2\le j\le N-1, \\
	b_{NN} &= \frac23\delta_N\Delta_N - \frac16(\omega_{N}^2-\omega_{N-1}^2)\delta_N
	- \frac19(2\omega_{N}^2 - \omega_{N-1}^2 - \omega_{N-1}\omega_{N})\delta_N
	- \beta_{NN},
\end{align*}
where
\begin{align*}
  \beta_{jj} &= - \frac 1{45} \omega_{j+1}^2 + \frac 1{90} \omega_{j+1}^2 \omega_j 
	+ \frac 1{18} \omega_{j+1}^2 \omega_{j+2} - \frac 1{15} \omega_{j+1} \omega_j^2 
	+ \frac 1 9 \omega_{j+2} \omega_{j+1} \omega_j \\
  &\phantom{xx}{}- \frac 19 \omega_{j+2}^2 \omega_{j+1} + \frac 7{90} \omega_j^3 
	- \frac 16 \omega_j^2 \omega_{j+2} + \frac 19 \omega_j \omega_{j+2}^2, \\
  \beta_{j+1,j} &= \frac 1 {45} \left(\omega_{j+2}^3 - \omega_{j+1}^3 
	+ 3(\omega_{j+1}^2 \omega_{j+2} - \omega_{j+2}^2 \omega_{j+1}) \right), \\
	\beta_{1N} &= \frac 1{45} \big(\omega_2^2 - \omega_1^2 + 3(\omega_2 \omega_1^2 
	- \omega_2^2 \omega_1)\big), \\
	\beta_{NN} &= \frac 1{45} \big(\omega_{N+1}^3 - \omega_N^3 
	+ 3(\omega_N^2 \omega_{N+1} - \omega_{N+1}^2 \omega_N)\big).
\end{align*}
\end{lemma}

\begin{proof}
The computation is similar to the previous proof. We write the integral for
$b_{jk}$ with $j$, $k\le N-1$ and for $j<k$ as
$$
  b_{jk} = M\int_0^M\phi_{N+j}(\eta)\,d\eta\int_0^M\phi_k(\eta')\,d\eta'
	- \int_0^M\phi_{N+j}(\eta)\,d\eta\int_0^M\eta'\phi_k(\eta')\,d\eta' - J^-_{jk},
$$
and for $j>k$ as
$$
  b_{jk} =  M\int_0^M\phi_{N+j}(\eta)\,d\eta\int_0^M\phi_k(\eta')\,d\eta'
	- \int_0^M\eta\phi_{N+j}(\eta)\,d\eta\int_0^M\phi_k(\etaÄ)\,d\eta' - J^+_{jk},
$$
where
\begin{align*}
  J^-_{jk} = \int_0^M\int_0^M(\eta-\eta')_+\phi_{N+j}(\eta)\phi_k(\eta')\,d\eta d\eta',
	\\
	J^+_{jk} = \int_0^M\int_0^M(\eta'-\eta)_+\phi_{N+j}(\eta)\phi_k(\eta')\,d\eta d\eta'.
\end{align*}
We compute
$$
  \int_0^M\phi_{N+j}(\eta)\,d\eta = \frac23\delta_j, \quad
	\int_0^M\eta\phi_{N+j}(\eta)\,d\eta = \frac13(\omega_{j+1}^2-\omega_j^2).
$$
The integrals $J^\pm_{jk}$ vanish if $k\neq j-1,j$ since $(\eta-\eta')_+$
and $(\eta'-\eta)_+$ vanish. This proves the expressions for $b_{jk}$ with
$j\le k-1$ and $j\ge k+2$.
The coefficients $b_{jj}$ and $b_{j+1,j}$ are calculated in the same way.

It remains to compute the matrix coefficients coming from the boundary elements.
The computation of $b_{jN}$ is straightforward as the support of $\phi_{2N}$ is
contained on the single subinterval $[\omega_{N-1},M]$. For the boundary
elements $b_{jN}$, we take into account that the support of $\phi_N$ is contained
in $[\omega_{N-1},M]$ and $[0,\omega_1]$, which yields
\begin{align*}
  b_{jN} &= M\int_0^M\phi_{N+j}(\eta)\,d\eta\int_0^M\phi_N(\eta')\,d\eta'
	- \int_0^M \eta\phi_{N+j}(\eta)\,d\eta\int_0^{\omega_1}\phi_N(\eta')\,d\eta'
	- \beta_{1N} - \beta_{NN},
\end{align*}
where
\begin{align*}
  \beta_{1N} &= \int_0^M\int_0^{\omega_1}(\eta'-\eta)_+\phi_{N+j}(\eta)
	\phi_N(\eta')\,d\eta d\eta', \\
	\beta_{NN} &= \int_0^M\int_{\omega_{N-1}}^M(\eta-\eta')_+\phi_j(\eta)
	\phi_N(\eta')\,d\eta d\eta',
\end{align*}
and the computation is as before.
\end{proof}

\begin{lemma}[Coefficients $c_{ij}$]\label{lem.c}
The coefficients of the symmetric matrix $C=(c_{ij})$, defined in \eqref{disc.abc},
read as
\begin{align*}
  c_{jj} &= \frac 89 \delta_j^2 - \frac 29 \delta_j (\omega_{j+1}^2 - \omega_j^2) 
	- \frac 2 {35} \big(\omega_{j+1}^3 - \omega_j^3 
	- 3\omega_{j+1}\omega_j(\omega_{j+1} - \omega_j)\big), \quad 1\le j\le N, \\
	c_{jk} &= \frac 89 \delta_j \delta_k - \frac 29 \delta_j (\omega_{k+1}^2 
	- \omega_k^2), \quad 1\le j<k\le N.
\end{align*}
\end{lemma}

\begin{proof}
We compute
\begin{align*}
  & c_{jk} = M\int_0^M\phi_{N+j}(\eta)\,d\eta\int_0^M\phi_{N+k}(\eta')\,d\eta'
	- \int_0^M\phi_{N+j}(\eta)\,d\eta\int_0^M\eta'\phi_{N+k}(\eta')\,d\eta'
	- \gamma_{jk}, \\
	&\mbox{where }\gamma_{jk} = \int_0^M\int_0^M(\eta-\eta')_+\phi_{N+j}(\eta)
	\phi_{N+k}(\eta')\,d\eta d\eta'.
\end{align*}
As before, we find that $\gamma_{jk}=0$ for all $j\neq k$. Moreover,
$$
  \gamma_{jj} = \int_{\omega_{j-1}}^{\omega_j}\phi_{N+j}(\eta)\,d\eta
	\int_{\omega_j}^\eta (\eta-\eta')\phi_{N+j}(\eta')\,d\eta'
	= \frac{2}{35}\big(\omega_{j+1}^3 - \omega_j^3 - 3(\omega_{j+1}\omega_j
	(\omega_{j+1}-\omega_j)\big),
$$
which finishes the proof.
\end{proof}

%%%%%%%%%%%%%%%

\subsection{Computation of the coefficients of the Hessian of $S_N$}\label{sec.hess}

We compute the gradient and Hessian of the discrete entropy \eqref{disc.SN}
for the case $\alpha=-1$. We set for $k=0,\ldots,N-1$:
$$
  S_{N,k}[{\mathbf g}] = \frac12\int_{\omega_k}^{\omega_{k+1}}
	\big(g_k\phi_k(\omega) + g_{k+1}\phi_{k+1}(\omega) + g_{N+k}\phi_{N+k}(\omega)\big)
  \,d\omega,
$$
where ${\mathbf g}=(g_1,\ldots,g_{2N})\in {\mathbb G}_M^N$.
Furthermore, we abbreviate $\pa_k S_{N,j}=\pa S_{N,j}/\pa g_k$
and $\pa_{j,k} S_{N,\ell} = \pa S_{N,\ell}/\pa g_j\pa_k g$.
A computation shows that
\begin{align*}
  \pa_k S_{N,k}[{\mathbf g}] &= \frac{\delta_{k+3}}{3}
	\big(2(g_{N+k}+g_k)+g_{k+1}\big), \\
	\pa_k S_{N,k-1}[{\mathbf g}] &= \frac23\delta_k(g_k-g_{k-1}-g_{N+k-1}), \\
	\pa_{N+k} S_{N,k}[{\mathbf g}] &= \delta_{k+1}\bigg((g_{k+1}+g_k)
	+ \frac85 g_{N+k}\bigg).
\end{align*}
As $S_{N,k}$ and $S_{N,k-1}$ depend on $g_k$, we obtain (recall \eqref{disc.SN})
$$
  \pa_k S_N = \pa_k S_{N,k} + \pa_k S_{N,k-1}, \quad
	\pa_{N+k} S_N = \pa_{N+k} S_{N,N+k}.
$$
The second-order derivatives become
\begin{align*}
  \pa_{k,k-1}S_{N,k-1} &= -\frac23\delta_k, &
	\pa_{k,k}S_{N,k-1} &= \frac23\delta_k, &
	\pa_{k,k}S_{N,k} &= -\frac23\delta_{k+1}, \\
	\pa_{k,k+1}S_{N,k} &= \frac13\delta_{k+1}, &
	\pa_{k,N+k-1}S_{N,k-1} &= -\frac23\delta_k, &
  \pa_{k,N+k}S_{N,k} &= \frac23\delta_{k+1}, \\
	\pa_{N+k,k+1}S_{N,k} &= \frac23\delta_{k+1}, &
	\pa_{N+k,N+k}S_{N,k} &= \frac{16}{15}\delta_{k+1}. & &
\end{align*}
Then the elements of the Hessian of $S_N$ read as
\begin{align*}
  \pa_{k,k-1}S_N &= \pa_{k,k-1}S_{N,k-1}, &
	\pa_{k,k}S_N &= \pa_{k,k}S_{N,k} + \pa_{k,k}S_{N,k-1}, \\
	\pa_{k,k+1}S_N &= \pa_{k,k+1}S_{N,k}, &
	\pa_{k,N+k}S_N &= \pa_{k,N+k}S_{N,k}, \\
	\pa_{k,N+k-1}S_N &= \pa_{k,N+k-1}S_{N,k-1}, &
	\pa_{N+k,k+1}S_N &= \pa_{N+k,k+1}S_{N,k}, \\
	\pa_{N+k,N+k}S_N &= \pa_{N+k,N+k}S_{N,k}.
\end{align*}
\end{appendix}

%%%%%%%%%%%%%%%%%%%%%%%%%%%%%%%%%%%%%%%%%%%%%%%%%%%%%%%%%%%%%%%%%%%%%%%%%%

\end{document}